\documentclass[letterpaper,11pt]{article}

\usepackage[margin=1in]{geometry}  
\usepackage{xspace,exscale,relsize}
\usepackage{fancybox,shadow}
\usepackage{graphicx}
\usepackage{color}
\usepackage{amsthm,amsfonts}
\usepackage{amssymb}
\usepackage{authblk}
\usepackage[title]{appendix}

\usepackage{extarrows}
\usepackage{comment}
\usepackage[noadjust]{cite}

\usepackage{amsmath}
	\makeatletter
	\let\over=\@@over \let\overwithdelims=\@@overwithdelims
	\let\atop=\@@atop \let\atopwithdelims=\@@atopwithdelims
  	\let\above=\@@above \let\abovewithdelims=\@@abovewithdelims
  	\makeatother
\usepackage{rotating}

\usepackage{ifpdf}

\usepackage{subfigure}
\usepackage{psfrag}

\usepackage{prettyref}
\usepackage{enumitem}

\usepackage{tikz}
\usetikzlibrary{arrows}
\tikzstyle{int}=[draw, fill=blue!20, minimum size=2em]
\tikzstyle{dot}=[circle, draw, fill=blue!20, minimum size=2em]
\tikzstyle{init} = [pin edge={to-,thin,black}]

\usepackage[
            CJKbookmarks=true,
            bookmarksnumbered=true,
            bookmarksopen=true,
            colorlinks=true,
            citecolor=red,
            linkcolor=blue,
            anchorcolor=red,
            urlcolor=blue,
            pdfauthor={Hang Du, Cheng Mao, Nike Sun, Yihong Wu, and Jiaming Xu}
            ]{hyperref}

\usepackage[all]{xy}

\usepackage{mathtools}

\usepackage{ifthen}
\newboolean{aos}
\setboolean{aos}{FALSE}

\newcommand{\KL}{\operatorname{KL}}
\newcommand{\Law}{\operatorname{Law}}

\ifx\eqref\undefined
	\newcommand{\eqref}[1]{~(\ref{#1})}
\fi
\ifx\mod\undefined
	\def\mod{\mathop{\rm mod}}
\fi

\usepackage{bm}

\newcommand{\norm}[1]{{\left\Vert #1 \right\Vert}}

\def\exp{\mathop{\rm exp}}

\def\tr{\mathop{\rm tr}}

\def\Var{\mathrm{Var}}

\def\diag{\mathop{\rm diag}}

\newcommand{\dd}{\mathrm d}

\newcommand{\Pbb}{\mathbb P}
\newcommand{\R}{\mathbb R}
\newcommand{\one}{\mathbf 1}

\newcommand{\ip}[2]{\left\langle #1,#2\right\rangle}

\newcommand{\Unif}{\mathrm{Uniform}}

\newcommand{\Expect}{\mathbb{E}}

\newcommand{\E}{\mathbb E}
\newcommand{\Prob}{\mathbb{P}}

\newcommand{\TV}{{\rm TV}}

\newcommand{\diff}{{\rm d}}

\newcommand{\Bern}{\text{Bern}}

\newcommand{\iprod}[2]{\left \langle #1, #2 \right\rangle}
\newcommand{\Iprod}[2]{\langle #1, #2 \rangle}

\definecolor{myblue}{rgb}{.8, .8, 1}
\definecolor{mathblue}{rgb}{0.2472, 0.24, 0.6} 
\definecolor{mathred}{rgb}{0.6, 0.24, 0.442893}
\definecolor{mathyellow}{rgb}{0.6, 0.547014, 0.24}

\newcommand{\red}{\color{red}}
\newcommand{\blue}{\color{blue}}
\newcommand{\nb}[1]{{\sf\blue[#1]}}
\newcommand{\nbr}[1]{{\sf\red[#1]}}
\newcommand{\JX}[1]{{\color{blue}[\textbf{JX:} #1]}}

\newcommand{\tW}{{\tilde{W}}}

\newcommand{\calG}{{\mathcal{G}}}
\newcommand{\calH}{{\mathcal{H}}}

\newrefformat{eq}{(\ref{#1})}
\newrefformat{thm}{Theorem~\ref{#1}}
\newrefformat{th}{Theorem~\ref{#1}}
\newrefformat{chap}{Chapter~\ref{#1}}
\newrefformat{sec}{Section~\ref{#1}}
\newrefformat{seca}{Section~\ref{#1}}
\newrefformat{algo}{Algorithm~\ref{#1}}
\newrefformat{fig}{Fig.~\ref{#1}}
\newrefformat{tab}{Table~\ref{#1}}
\newrefformat{rmk}{Remark~\ref{#1}}
\newrefformat{clm}{Claim~\ref{#1}}
\newrefformat{def}{Definition~\ref{#1}}
\newrefformat{cor}{Corollary~\ref{#1}}
\newrefformat{lmm}{Lemma~\ref{#1}}
\newrefformat{prop}{Proposition~\ref{#1}}
\newrefformat{pr}{Proposition~\ref{#1}}
\newrefformat{app}{Appendix~\ref{#1}}
\newrefformat{apx}{Appendix~\ref{#1}}
\newrefformat{ex}{Example~\ref{#1}}
\newrefformat{exer}{Exercise~\ref{#1}}
\newrefformat{soln}{Solution~\ref{#1}}

\def\unifto{\mathop{{\mskip 3mu plus 2mu minus 1mu%
	\setbox0=\hbox{$\mathchar"3221$}%
	\raise.6ex\copy0\kern-\wd0%
	\lower0.5ex\hbox{$\mathchar"3221$}}\mskip 3mu plus 2mu minus 1mu}}

\ifx\lesssim\undefined
\def\simleq{{{\mskip 3mu plus 2mu minus 1mu%
	\setbox0=\hbox{$\mathchar"013C$}%
	\raise.2ex\copy0\kern-\wd0%
	\lower0.9ex\hbox{$\mathchar"0218$}}\mskip 3mu plus 2mu minus 1mu}}
\else
\def\simleq{\lesssim}
\fi

\ifx\gtrsim\undefined
\def\simgeq{{{\mskip 3mu plus 2mu minus 1mu%
	\setbox0=\hbox{$\mathchar"013E$}%
	\raise.2ex\copy0\kern-\wd0%
	\lower0.9ex\hbox{$\mathchar"0218$}}\mskip 3mu plus 2mu minus 1mu}}
\else
\def\simgeq{\gtrsim}
\fi

\newtheorem{theorem}{Theorem}
\newtheorem{lemma}[theorem]{Lemma}
\newtheorem{corollary}[theorem]{Corollary}

\newtheorem{proposition}[theorem]{Proposition}

\theoremstyle{definition}

\newtheorem{remark}{Remark}

\newif\ifmapx
{\catcode`/=0 \catcode`\\=12/gdef/mkillslash\#1{#1}}
\edef\jobnametmp{\expandafter\string\csname embayes2_apx\endcsname}
\edef\jobnameapx{\expandafter\mkillslash\jobnametmp}
\edef\jobnameexpand{\jobname}
\ifx\jobnameexpand\jobnameapx
\mapxtrue
\else
\mapxfalse
\fi

\newcommand{\polylog}{\mathsf{polylog}}

\renewcommand{\tilde}{\widetilde}

\hypersetup{
	colorlinks   = true, 
	urlcolor     = blue, 
	linkcolor    = blue, 
	citecolor   = red 
}

\newcommand{\HD}[1]{{\color{red}[\textbf{HD:} #1]}}

\newcommand{\ER}{Erd\H{o}s-R\'enyi\ }

\newcommand{\Pw}{\mathbb{P}_{0}}

\newcommand{\PRGG}{\mathbb{P}_{\operatorname{RGG}}}
\newcommand{\PER}{\mathbb{P}_{\operatorname{ER}}}

\begin{document}

	\title{Resolution of the Detection Threshold Conjecture for Random Geometric Graphs in the $d>n$ Regime}
	
	\author{Hang Du, Cheng Mao, Nike Sun, Yihong Wu, and Jiaming Xu\thanks{
	H.\ Du and N.\ Sun are with the Department of Mathematics, Massachusetts Institute of Technology, Cambridge, Massachusetts, USA, \texttt{\{hangdu,nsun\}@mit.edu}.
	C.\ Mao is with the School of Mathematics, Georgia Institute of Technology, Atlanta, Georgia, USA, \texttt{cheng.mao@math.gatech.edu}.
	Y.\ Wu is with the Department of Statistics and Data Science, Yale University, New Haven, Connecticut, USA, \texttt{yihong.wu@yale.edu}.
	J.\ Xu is with The Fuqua School of Business, Duke University, Durham, North Carolina, USA, \texttt{jx77@duke.edu}.
	}}
	
	\date{\today}
	
	\maketitle


\begin{abstract}

A random geometric graph (RGG) is generated by first sampling latent points
$x_1,\ldots,x_n$ independently and uniformly from the unit
sphere in $\R^d$, and then connecting each pair $(i,j)$ if
$\langle x_i,x_j\rangle$ exceeds some threshold $\tau$. We study the sharp detection threshold---the
largest dimension at which the RGG can be statistically distinguished from the
Erd\H{o}s--R\'enyi graph with the same edge density $p$. This threshold is
conjectured to be $d \asymp (nh(p))^3$, where $h(p)=p \log \frac{1}{p} + (1-p) \log \frac{1}{1-p}$ is the binary entropy
function. Previous works proved this conjecture for dense graphs with constant $p$ and, up to
polylogarithmic factors, very sparse graphs with $p=\Theta(1/n)$. In this paper, we prove that
detection is impossible when $d\gg (nh(p))^3$ and $d\ge (1+\epsilon) n$ for any constant $\epsilon>0$, thereby resolving the conjecture in the regime $p\gtrsim n^{-2/3}/\log n$ and improving upon the state of the art in the regime $1/n \ll p \ll n^{-2/3}/\log n$.

The key to our proof is a sharp analysis of the posterior distribution of the
latent points given the observed graph, obtained through an information-theoretic
comparison argument combined with strong log-concavity.
\end{abstract}

\tableofcontents

\section{Introduction}

Random geometric graphs (RGGs) are network models in which edges are generated from latent geometric positions~\cite{penrose2003random,duchemin2023random}. They arise naturally in wireless network models, beginning with the Gilbert disk model~\cite{Gilbert61},
and belong to the broader class of latent space models for relational data~\cite{hoff2002latent,KaurRastelliFrielRaftery2023LPM}. In this paper, we study the high-dimensional hard-thresholded RGG model.

Fix $n,d\in\mathbb N$ and $p\in(0,1)$. Let $x_1,\ldots,x_n$ be independently and uniformly from the unit sphere
from $\mathbb S^{d-1}$, and 
let  $X$ be the
$n\times d$ matrix whose $i$th row is $x_i^\top$,
Choose the threshold
$\tau\equiv\tau(d,p)$ so that
$
\mathbb P(\langle x_1,x_2\rangle\ge \tau)=p
$.
The hard RGG $G(X)$ is the graph on $[n]$ in which distinct vertices $i,j$ are
adjacent if and only if $\langle x_i,x_j\rangle\ge \tau$. We denote its law by $\PRGG$.

A central question for high-dimensional RGGs is whether the latent geometry can
be detected from the observed graph itself. Let $\PER$ denote the law of the
Erd\H{o}s--R\'enyi graph $G(n,p)$. The two models $\PRGG$ and $\PER$ have the
same single-edge marginal distribution, but edges in an RGG are dependent
through the latent positions. Detecting geometry is therefore naturally
formulated as the hypothesis testing problem of distinguishing $\PRGG$ from
$\PER$, measured by the total variation distance
$\operatorname{TV}(\PRGG,\PER)$: As $n\to \infty$, if it tends to zero, no test can asymptotically outperform random guess, whereas if it tends to one, the two models can be distinguished with vanishing error probability.\footnote{Recall that for discrete distributions $P$ and $Q$, the total variation is 
$\TV(P,Q) \equiv \frac{1}{2} \sum_a |P(a)-Q(a)|$ and $1-\TV(P,Q)$ is the minimal total error probability for distinguishing $P$ and $Q$, attained by the likelihood ratio test. In addition, the Kullback-Leibler divergence is defined as 
$\KL(P\|Q) \equiv \sum_a P(a)\log\frac{P(a)}{Q(a)}$ if $P\ll Q$ and $\infty$ otherwise.} The goal is to identify the
detection threshold, the highest dimension, at which the
latent geometry can be statistically detectable from the observed graph.

This detection problem for hard RGG was studied by the seminal work~\cite{bubeck2016testing}, which proved the detection threshold $d \asymp n^3$
for dense graphs with constant $p$. Since then, a substantial body of work has
studied the detection for hard RGGs~\cite{brennan2020phase,liu2022testing,bangachev2024fourier,bangachev2025sandwiching}
and for the so-called soft RGGs with smooth kernels~\cite{liu2023probabilistic,mao2026random}. A major open problem is to determine the detection threshold for hard RGGs in the sparse regime, where the edge density $p=p(n)$ vanishes as $n\to\infty$. 
 This is the focus of the present paper.

The positive direction of the problem is well-understood: the signed triangle count
distinguishes $\PRGG$ from $\PER$ when
$
d\ll n^3 h(p)^3,
$
where $h(p)=p \log \frac{1}{p} + (1-p) \log \frac{1}{1-p}$ is the binary entropy function~\cite{bubeck2016testing,liu2022testing}. Remarkably, this
threshold is conjectured to be information-theoretically optimal. However, on the negative side, the picture is much less complete. It is known that 
$\operatorname{TV}(\PRGG,\PER)\to0$ when
\begin{itemize}
\item $d\gg n^3$ for constant
$p\in(0,1)$~\cite{bubeck2016testing};
\item $d\gg n^3p^2\polylog(n)$ for $1/n\ll p\ll1$~\cite{liu2022testing};
\item $d\gg (\log(n))^{36}$ for
$p=\Theta(1/n)$~\cite{liu2022testing}.
\end{itemize} 
In terms of computational limit, \cite{bangachev2024fourier} proved that  
no test statistic that is a degree-$(\log(n))^{1.1}$ polynomial 
can distinguish the two models with  constant advantage
when $d\ge (np)^{3+\epsilon}$ for any
fixed $\epsilon>0$. 
Nevertheless, whether detection is information-theoretically impossible when
$d\gg (nh(p))^3$ has remained open for $p=o(1)$.

In this paper, we prove this conjectured threshold for $h(p) \gtrsim n^{-2/3}$, i.e., $p\gtrsim n^{-2/3}/\log n$.  

\begin{theorem}[Main result] \label{thm:main-detection}
For any constant $\epsilon>0$, if $\epsilon/n \le p \le 1-\epsilon$, $d \ge (1+\epsilon) n$, and $d\gg (n h(p))^3$,
then 
$$\operatorname{KL}(\PRGG\|\PER)=o(1) .
$$
It then follows from Pinsker's inequality that $\operatorname{TV}(\PRGG,\PER) =o(1)$.
\end{theorem}

The \hyperlink{proof.main.thm}{proof of \prettyref{thm:main-detection}} is presented at the end of Section~\ref{s:proof.main.thm}.
In the regime $h(p)\gtrsim n^{-2/3}$, equivalently
$p\gtrsim n^{-2/3}/\log n$, we have $(nh(p))^3\gtrsim n$, so
\prettyref{thm:main-detection} implies the impossibility of detection whenever $
    d\gg (nh(p))^3 $.
Combined with the existing achievability result based on the signed triangle
count~\cite{bubeck2016testing,liu2022testing}, this establishes the conjectured detection threshold in the regime
$h(p)\gtrsim n^{-2/3}$. Extending the proof to $d \le n$, and thereby
establishing the conjectured threshold for sparser graphs with
$h(p)\ll n^{-2/3}$, remains open. See~\prettyref{sec:conclusion} for
discussion of the challenges. 
Nevertheless, in the regime $1/n\ll p\ll n^{-2/3}/\log n$, we have $(nh(p))^3 \ll n$ and thus our impossibility condition becomes $d\ge (1+\epsilon) n$, which still improves the 
state-of-the-art condition $d\gg n^3p^2\polylog(n)$~\cite{liu2022testing}.

\subsection{Notations}
We collect some notations used throughout this paper.  For $x_1,\dots,x_n \in \mathbb{S}^{d-1}$, recall that $X$ denotes the $n\times d$ matrix with rows $x_i^\top$, and that $G=G(X)$ denotes the random geometric graph with edges $(i,j)$ such that $\langle x_i,x_j\rangle\ge\tau$.
Let $\Pw$ denote the prior distribution of $X$. For a graph $G$, denote $\mu_{G}$ as the posterior distribution of $X$ given that $G(X)=G$.
Define the $n\times n$ matrix 
   $W=XX^\top-I_n$, with
     zero diagonal and off-diagonal entries 
     $W_{ij}=\langle x_i,x_j\rangle$ for $i\ne j$.
    For ease of notation, for any two symmetric, zero-diagonal matrices   $W,W^*\in \mathbb{R}^{n\times n}$, we write
 \[
 \langle W,W^*\rangle:=\sum_{1\le i<j\le n} W_{ij} W^*_{ij},
\qquad
\|W \|^2:=\langle W,W\rangle.
\]
For a vector $x$, we also use $\|x \|$ to denote its Euclidean norm.

\section{Proof of Theorem~\ref{thm:main-detection}}
\label{s:proof.main.thm}

The main goal of this section is to present the \hyperlink{proof.main.thm}{proof of \prettyref{thm:main-detection}}. The proof proceeds in three steps.
\begin{itemize}
    \item 
In \prettyref{sec:KL}, we review a KL expansion from \cite{mao2026random}: using the chain rule for conditional entropies, the KL divergence from
 the RGG to the \ER graph can be expanded in terms of the posterior moments of certain \textbf{local interaction functions} of the latent points given the RGG. 
    \item 
    In \prettyref{sec:etak}, we derive expansions of these local interaction functions as polynomials with respect to inner products of the input vectors; see \prettyref{thm:taylor-eta2345}. 
    This reduces the problem to analyzing moments of the posterior overlap 
$\Iprod{W}{W^*}$, where $W,W^*$ are two independent replicas drawn from the posterior law $\mu_G$.

    \item In 
    \prettyref{sec:posterior-main} we control the posterior moments.     
    The first moment of the posterior overlap is bounded using an \textbf{information-theoretic} argument in terms of the entropy of the RGG graph; see \prettyref{thm:posterior_mean}.
 To bound higher moments, 
we crucially exploit the \textbf{strong log-concavity} of the prior and the posterior; see \prettyref{thm:posterior-overlap-moments}.
Specifically, the posterior of $W$ given $G$, after appropriate truncation, is strongly log-concave with curvature $\Theta(d)$. The concentration inequality for strongly log-concave distributions then allows us to suitably control the higher moments. We remark that this is the part where $d>(1+\epsilon)n$ becomes essential. 

\end{itemize}



\subsection{KL expansion}
\label{sec:KL}
To bound the joint KL divergence between the RGG and $G(n,p)$, applying the chain rule one node at a time yields
\[
\KL(\PRGG\|\PER) = \sum_{k=2}^{n}
\Expect_{\text{RGG}}[\KL(\Prob_{A_{k1},\ldots,A_{k,k-1}|A^{k-1}}\|\Bern(p)^{\otimes (k-1)})]
\]
where $A=(A_{ij})$ is the adjacency matrix of the RGG and $A^{k-1}$ is that of the subgraph induced by the first $k-1$ nodes.
Recognizing the conditional law \[\Prob_{A_{k1},\ldots,A_{k,k-1}|A^{k-1}}\] as a mixture of Bernoulli products and bounding each conditional KL by the $\chi^2$-divergence, \cite[Lemma~5]{mao2026random} showed the upper bound
\begin{equation}\label{eq:KL}
\KL(\PRGG\|\PER)\le \sum_{k=2}^{n-1}\binom{n}{k+1}g(k)\,,
\end{equation}
where $g(k)$ is defined by
\begin{equation}
g(k) \triangleq \mathbb{E}_{G\sim \PRGG}\left[\left(\mathbb{E}_{X\sim \mu_G}\big[\eta_k(x_1,\dots,x_k)\big]\right)^2\right]\,,
\label{eq:gk}
\end{equation}
for $\eta_k:(\mathbb{S}^{d-1})^k\to\mathbb{R}$ defined by
\begin{equation}\label{eq:eta-def}
\eta_k(x_1,\dots,x_k)\triangleq\mathbb{E}_{x\sim \operatorname{Unif}(\mathbb{S}^{d-1})}\left[\prod_{i=1}^k\kappa(\langle x,x_i\rangle)\right]\,, 
\quad \kappa(t) \triangleq \frac{\mathbf{1}\{t\ge \tau\}-p}{\sqrt{p(1-p)}}\,.
\end{equation}
We refer to the $\eta_k$ as \textbf{local interaction functions}, since each $\eta_k$ depends on the Gram matrix $(W_{ij})_{1\le i<j\le k}$ of $x_1,\ldots,x_k$; and we will see that the main contribution turns out to come from small $k$. Nevertheless, we note that this does not reduce the calculation to a RGG with $k$ nodes since each $g(k)$ still requires averaging over the posterior given the full RGG.

\subsection{Estimates on local interaction function}
\label{sec:etak}

It turns out that to yield the desired threshold in \prettyref{thm:main-detection} for $d\gtrsim n$, we only need precise control of the local interaction functions $\eta_k$ for $k \leq 5$, while a crude bound suffices for higher-order terms. To this end, define
\begin{equation}\label{eq:delta-def}
\delta_k = \delta_k(x_1, \dots, x_k)
\triangleq
\frac{1}{d^{1/2}}+\max_{1 \le i<j \le k}|W_{ij}|\,;
\end{equation}
this is a quantity we use to control the remainder terms in the Taylor expansion of $\eta_k$. 
For $x,x'$ drawn independently and uniformly from $S^{d-1}$, define
\begin{align}
T=\sqrt{d} \iprod{x}{x'}, \quad a= \tau \sqrt{d}, \quad \text{ and } \quad \zeta= \bigg( \frac{p}{1-p}\bigg)^{1/2}   \Expect[T | T\geq a]\,.
\label{eq:zeta-def}
\end{align}
We have 
$1\le a^2\lesssim \log (1/p)$ and $\zeta^2 \asymp  pa^2
\lesssim h(p)$
(cf.~\prettyref{lmm:m1m2}).


\begin{theorem}
\label{thm:taylor-eta2345}
There exist numerical constants $c,C>0$ such that the following holds. For $2\le k\le 5$, define the functions
\begin{align*}
E_2(x_1,x_2) 
&\triangleq \zeta^2\mathsf{S}_{2,1}\,,\\
E_3(x_1,x_2,x_3)
&\triangleq (1-1/d) a\zeta^3 \mathsf{S}_{3,2}\,,\\
E_4(x_1,\ldots,x_4)
&\triangleq \zeta^4 \mathsf{S}_{4,2}\,,\\
E_5(x_1,\ldots,x_5)
&\triangleq (1-1/d)a\zeta^5 \mathsf{S}_{5,3}\,,
\end{align*}
where $\mathsf{S}_{v,e}$ is the weighted sum over all graphs on vertex set $[v]$ with $e$ edges, such that each vertex is covered by at least one edge:
\begin{align*}
\mathsf{S}_{2,1} &= W_{12}\,,\\
\mathsf{S}_{3,2} &= W_{12} W_{13}
    +W_{12} W_{23}+W_{13} W_{23}\,,\\
\mathsf{S}_{4,2} &= W_{12} W_{34}
+
W_{13} W_{24}
+
W_{14} W_{23}\,,\\
\mathsf{S}_{5,3} &= \sum_{h=1}^5
\sum_{\{i,j\}\subset [5]\setminus\{h\}}
W_{hi}W_{hj}W_{uv},
\qquad
\{u,v\}=[5]\setminus\{h,i,j\}\,.
\end{align*}
Define the residuals $R_k\equiv \eta_k-E_k$. On the event $a^2\delta_k \le c$, 
    \[|R_k| \le C \left( \frac{p}{1-p}\right)^{k/2} (a^2\delta_k)^{\lceil k/2\rceil +1}\]
for all $2\le k\le 5$.
Additionally, we have
$$
|\eta_5(x_1,\ldots,x_5)|
\le C \left(\frac{p}{1-p}\right)^{5/2}(a^2\delta_5)^3
$$ as an overall bound for $k=5$.
\end{theorem}

The \hyperlink{proof:t.taylor-eta2345}{proof of \prettyref{thm:taylor-eta2345}} relies on calculations specific to the hard RGG model, and is postponed to \prettyref{app:etak}. To provide some intuition, we give here a \textbf{heuristic} derivation of the expansion for $\eta_2$: for $T_i=d^{1/2}\langle x,x_i\rangle$, we will use the approximation that $(T_1,T_2)$ is bivariate normal with mean zero, $\Var(T_i)=1$, and covariance $\E(T_1 T_2)=W_{12}$. As a result, the distribution of $T_2$ conditional on $T_1$ is roughly $W_{12} T_1 + (1-(W_{12})^2)^{1/2} Z$, where $Z$ is a standard gaussian independent of $T_1$. Thus, if $\Psi(x) \equiv \mathbb{P}(Z\ge x)$ denotes the complementary gaussian cumulative distribution, we can use a Taylor expansion to approximate
    \[
    \mathbb{P}(T_2\ge a \,|\, T_1)-p
    \approx \Psi\bigg(
    \frac{a-W_{12}T_1}{(1-(W_{12})^2)^{1/2}}\bigg)
    -p
    \approx \varphi(a)W_{12} T_1
    \asymp p a W_{12} T_1\,.
    \]
Taking expectation over $T_1$ gives
    \[
    \mathbb{P}(T_2\ge a \,|\, T_1\ge a)-p
    \asymp pa W_{12} \E(T_1\,|\,T_1\ge a)
    \asymp p^{1/2} a W_{12} \zeta
    \asymp \zeta^2 W_{12}\,,
    \]
recalling that $\zeta\asymp p^{1/2}a$. Altogether this heuristic derivation gives
    \[
    \eta_2(x_1,x_2)
    =\frac{\mathbb{P}(T_1\ge a,T_2\ge a)-p^2}{p(1-p)}
    = \frac{\mathbb{P}(T_2\ge a \,|\, T_1\ge a)-p}{1-p}
    \asymp \zeta^2 W_{12}\,,
    \]
which is consistent with the assertion of Theorem~\ref{thm:taylor-eta2345}. For the rigorous proof we work with the spherical distributions rather than the gaussian approximation, and use an interpolation argument to control the residual term; these arguments are deferred to \prettyref{app:etak}.


\subsection{Posterior overlap and concentration}
\label{sec:posterior-main}

Thanks to  \prettyref{thm:taylor-eta2345}, each $g(k)$ in \prettyref{eq:gk} 
can be approximated by certain posterior moments of $W$ under $\mu_G$. For example, the expansion of $\eta_2$ from \prettyref{thm:taylor-eta2345} tells us that
    \begin{align*}
    g(2) &\approx \zeta^4
    \Expect_{G\sim \PRGG}
    \Big[(\Expect[W_{12}|G])^2\Big] 
    = \frac{\zeta^4}{\binom{n}{2}}
    \Expect_{G\sim \PRGG}
    \Big[\|\Expect[W|G]\|^2\Big] \\
    &= \frac{\zeta^4}{\binom{n}{2}}
    \Expect_{G\sim \PRGG}\Big[
    \Expect_{(W,W^*)\sim(\mu_G)^{\otimes2}}
    [\langle W,W^* \rangle]
    \Big]
    \,,
    \end{align*}
where the first equality holds by symmetry.
This quantity is bounded by the following result.

\begin{theorem}\label{thm:posterior_mean}
If $d\ge (1+\epsilon)n$ for any constant $\epsilon>0$ and 
    $nh(p) \geq 1$, then  
    \[\Expect_{G\sim \PRGG}\Big[
        \|\mathbb{E}_{X \sim \mu_G}[W]\|^2\Big]
    \leq C \frac{n^2h(p)}{d} 
     +\frac1{n^{\omega(1)}} \,,
    \]
    where $C$ is a constant depending only on $\epsilon$.
\end{theorem}

As noted above, the result of
\prettyref{thm:posterior_mean} is equivalent to a bound on the first moment of the posterior overlap $\Iprod{W}{W^*}$, where $W,W^*$ are i.i.d.\ samples from the posterior distribution $\mu_G$. The next result bounds the higher moments of this overlap, which is useful for bounding $g(k)$ for $k\ge3$. For example, the second moment of the posterior overlap is
    \begin{align*}
    &\Expect_{G\sim \PRGG}\Big[\mathbb{E}_{(W,W^*)\sim (\mu_G)^{\otimes 2}} [\langle W,W^*\rangle^2]\Big]\\
    &\qquad= 
    \sum_{i<j}\sum_{k<\ell}
    \Expect_{G\sim \PRGG}\Big[
    \Expect_{(W,W^*)\sim (\mu_G)^{\otimes 2}}
    [
    W_{ij}(W^*)_{ij} W_{kl} (W^*)_{kl}]\Big]\\
    &\qquad= \sum_{i<j}\sum_{k<\ell}\Big[
    \Big(\mathbb{E}_{G\sim \PRGG} [W_{ij}W_{kl}]\Big)^2
    \Big]\,.
    \end{align*}
It is easy to see that the last expression is related to the posterior second moment of $\mathsf{S}_{3,2}$ and $\mathsf{S}_{4,2}$, which is related to $g(3)$ and $g(4)$. (For details see the proof of \prettyref{lem:partial-sum-2345}.)

\begin{theorem}\label{thm:posterior-overlap-moments}
        For $d\ge (1+\epsilon)n$, $n h(p) \ge 1$, and any integer $k\ge 1$, we have 
    \[
    \mathbb{E}_{G\sim \PRGG}\Big[\mathbb{E}_{(W,W^*)\sim (\mu_G)^{\otimes 2}} [\langle W,W^*\rangle^{2k}]\Big]
    \le C
    \bigg(\frac{n^2h(p)}{d}\bigg)^{2k} + \frac1{n^{\omega(1)}}\,,
    \]
where $C$ denotes a constant depending only on $k$ and $\epsilon$.
\end{theorem}
Unlike \prettyref{thm:taylor-eta2345} which is specific to the hard RGG model, the analysis for the posterior distribution is far more general. In fact, \prettyref{thm:posterior_mean} holds for any model on $G$ with marginal edge density $p$, and \prettyref{thm:posterior-overlap-moments} only relies on the strong log-concavity of the posterior distribution $\mu_G$. The \hyperlink{proof:t.posterior_mean}{proof of Theorem~\ref{thm:posterior_mean}} is given in \prettyref{sec:posterior_mean}. 
The \hyperlink{proof:t.posterior-overlap-moments}{proof of Theorem~\ref{thm:posterior-overlap-moments}} is given in \prettyref{sec:concentration}.

We remark that \prettyref{thm:posterior_mean} has a natural interpretation from the perspective of estimation. Due to symmetry, given the RGG, the latent vectors $x_i$ are only identified up to a common rotation, so it is natural to consider estimation of the  Gram matrix $W$. Note that 
\[\mathbb{E}_{G\sim \PRGG}\Big[\|\mathbb{E}_{X \sim \mu_G}[W]\|^2\Big]
= \sum_{1 \leq i < j \leq n}
\mathbb{E}_{G\sim \PRGG}\Big[(\mathbb{E}[W_{ij} |G])^2\Big]
\]
precisely captures the variance reduction in $W$ resulting from the observation of the RGG. It is straightforward to check that observing only the edge $A_{ij}$ already yields
\[\Expect\Big[\Expect[W_{ij}|A_{ij}])^2\Big]
= \frac{\zeta^2}{d} \asymp \frac{h(p)}{d}\,.\]
Thus, \prettyref{thm:posterior_mean} states that observing the full RGG can reduce the variance further by at most a constant factor. 

\subsection{Completing the proof of~\prettyref{thm:main-detection}}
For $p$ bounded away from $0$ and $1$, the TV bound is a seminal result of \cite{bubeck2016testing}, and the KL bound follows from \cite{brennan2020phase}
(see also 
\cite[Equation~(31)]{mao2026random}). 
Thus, we may assume $p \ll 1$ in the following.

To prove Theorem~\ref{thm:main-detection}, we use the KL expansion \prettyref{eq:KL}, where each $g(k)$ in \prettyref{eq:gk} involves the posterior mean of the local interaction function $\eta_k$ defined in \prettyref{eq:eta-def}. 
We then apply \prettyref{thm:taylor-eta2345} to approximate each $\eta_k$. The contributions of $g(k)$ for $2\le k\le 5$ are controlled using the posterior moment bounds in Theorems~\ref{thm:posterior_mean} and~\ref{thm:posterior-overlap-moments}; this is carried out in Lemma~\ref{lem:partial-sum-2345}. For $k\ge 6$, we apply a crude bound using Jensen's inequality, as shown in Lemma~\ref{lem:partial-sum-ge-6}. Combining these two lemmas then proves \prettyref{thm:main-detection}.

\begin{lemma} \label{lem:partial-sum-2345}
Assume that $d\gg (nh(p))^3$, $d \ge (1+\epsilon) n$, and $1/n \lesssim p \ll 1$. 
Then we have 
$$
\sum_{k = 2}^{5} \binom{n}{k+1} g(k) = o(1) .
$$
\end{lemma}

\begin{proof}
Let $c_0$ denote the constant $c$ in
\prettyref{thm:taylor-eta2345}.  For $2 \le k \le 5$, define the good event
\[
\mathcal A_k=\mathcal A_k(x_1,\ldots,x_k)
\triangleq 
\Big\{a^2\delta_k(x_1,\ldots,x_k)\le c_0\Big\}\,.
\]
Using that $1/n \lesssim p \ll 1$ and $d \gg (nh(p))^3$, we have (cf.\ \prettyref{lmm:m1m2})
    \[
    1 \le a^2 \lesssim \log\frac{1}{p} \ll d^{1/2}\,,
    \]
from which it follows that
\[
(\mathcal A_k)^c
=
\bigg\{
\max_{1\le i<j\le k}|W_{ij}|>\frac{c_0}{a^2}
-\frac{1}{d^{1/2}}
\bigg\}
\subseteq
\bigg\{
\max_{1\le i<j\le k}|W_{ij}|>\frac{c_0}{2a^2}
\bigg\}\,.
\]
It is well known that $d^{1/2}W_{12}$ is $1$-subgaussian (see \prettyref{lem:beta.mgf}), so we conclude
\begin{equation}\label{e:prob.of.good.event.A.k}
\Pw((\mathcal A_k)^c)
\le C_k \exp\bigg(-\frac{c_k d}{a^4}\bigg)
=n^{-\omega(1)}\,,
\end{equation}
where the last bound uses $p \gtrsim 1/n$ and $d \gg (nh(p))^3$. Following the notation of \prettyref{thm:taylor-eta2345}, for $2\le k\le 5$ we can decompose
    \[\eta_k = E_k 
    + R_k \mathbf 1_{\mathcal{A}_k}
    + R_k \mathbf 1_{(\mathcal{A}_k)^c}\,.
    \]
Substituting this into the definition \prettyref{eq:gk} of $g(k)$, we obtain
    \begin{align*}
    \frac{g(k)}{3} 
    &\le 
    g(k)_\textup{main}
    +g(k)_\textup{res}+g(k)_\textup{bad}\,,\\
    g(k)_\textup{main}
    &\equiv \mathbb{E}_{G\sim \PRGG}\left[
\left(\mathbb{E}_{X\sim\mu_G}[E_k]\right)^2\right]\,,\\
    g(k)_\textup{res}
    &\equiv \mathbb{E}_{G\sim \PRGG}\left[
\left(\mathbb{E}_{X\sim\mu_G}[
    R_k \mathbf{1}_{\mathcal{A}_k}
    ]\right)^2\right]\,,\\
    g(k)_\textup{bad}
    &\equiv \mathbb{E}_{G\sim \PRGG}\left[
\left(\mathbb{E}_{X\sim\mu_G}[R_k
    \mathbf{1}_{(\mathcal{A}_k)^c}
    ]\right)^2\right]\,.
    \end{align*}
We bound each of the above contributions separately. 

\paragraph{Contribution from bad events $(\mathcal{A}_k)^c$.}
For the last quantity above, recalling that $|W_{ij}|\le1$ and $\zeta^2 \le O(h(p))$ (cf.\ \prettyref{lmm:m1m2}), it is easy to see that $|\eta_k| + |E_k| \le n^{O(1)}$. It follows using Jensen's inequality that
\begin{align*}
g(k)_\textup{bad}
&=\mathbb{E}_{G\sim \PRGG}\left[
\left(\mathbb{E}_{X\sim\mu_G}
[(\eta_k-E_k)\mathbf 1_{(\mathcal A_k)^c}]\right)^2\right]
\le n^{O(1)} \mathbb{E}_{G\sim \PRGG}\left[
\left(\mathbb{E}_{X\sim\mu_G}[\mathbf 1_{(\mathcal A_k)^c}]\right)^2
\right] \\
&\qquad\le n^{O(1)} \mathbb{E}_{G\sim \PRGG}
\mathbb{E}_{X\sim\mu_G}[\mathbf 1_{(\mathcal A_k)^c}]
=n^{O(1)}\Pw((\mathcal A_k)^c)
=\frac1{n^{\omega(1)}}\,,
\end{align*}
where the last bound holds due to 
\eqref{e:prob.of.good.event.A.k}. 

\paragraph{Contribution from residuals on good events $\mathcal{A}_k$.}
Applying Jensen's inequality again gives
    \begin{align*}
    g(k)_\textup{res}
    &\le 
\mathbb{E}_{G\sim \PRGG}
\mathbb{E}_{X\sim\mu_G}[(R_k)^2
    \mathbf{1}_{\mathcal A_k}]
= \E_{\Pw}[(R_k)^2  \mathbf{1}_{\mathcal A_k} ]\\
&\lesssim p^k a^{4(\lceil k/2\rceil+1)}
\E_{\Pw}\bigg[
\max_{1\le i,j \le k}
\bigg(|W_{ij}|+\frac{1}{d^{1/2}}\bigg)^{2(\lceil k/2\rceil+1)}\bigg]
\lesssim \frac{p^k}{d^{\lceil k/2 \rceil+1}}
    \bigg(\log \frac{1}{p}\bigg)^{2(\lceil k/2 \rceil+1)}\,.
    \end{align*}
This results in
\begin{align*}
    \binom{n}{3}g(2)_\textup{res}
    &\lesssim \frac{n^3 p^2}{d^2}
    \bigg(\log \frac{1}{p}\bigg)^4
    \lesssim \frac{n^{3/2} p^2}{d^{1/2}}
    \bigg(\log \frac{1}{p}\bigg)^4
    \ll p^{1/2}
    \bigg(\log \frac{1}{p}\bigg)^4
    = o(1) ,\\
    \binom{n}{4}g(3)_\textup{res}
    &\lesssim \frac{n^4p^3}{d^3}
    \bigg(\log \frac{1}{p}\bigg)^6
    \lesssim \frac{n^{3/2}p^3}{d^{1/2}}
    \bigg(\log \frac{1}{p}\bigg)^6
    \ll p^{3/2}
    \bigg(\log \frac{1}{p}\bigg)^6
    =o(1) ,\\
    \binom{n}{5}g(4)_\textup{res}
    &\lesssim \frac{n^5p^4}{d^3}
     \bigg(\log \frac{1}{p}\bigg)^6
     \lesssim \frac{n^3p^4}{d}
     \bigg(\log \frac{1}{p}\bigg)^6
     \ll p
     \bigg(\log \frac{1}{p}\bigg)^6
     =o(1) ,
    \end{align*}
where in each line above, we have used $d \gtrsim n$ for the second inequality and $d \gg (n h(p))^3$ for the third inequality.

\paragraph{Contribution from main terms.}
It remains to bound the contribution from the leading terms $E_k$.
For $k=2$, again recalling $\zeta^2 = O(h(p))$, we have
    \begin{align*}
    \binom{n}{3} g(2)_\textup{main} 
    &= \binom{n}{3} \mathbb{E}_{G\sim \PRGG}\left[
\left(\mathbb{E}_{X\sim\mu_G}[\zeta^2W_{12}]\right)^2\right] \\
&=
\frac{\binom{n}{3}\zeta^4}{\binom{n}{2}}
\mathbb{E}_{G\sim \PRGG}
\Big[
\|\mathbb{E}_{X\sim \mu_{G}}[W]\|^2 \Big]
\lesssim 
\frac{(nh(p))^3}{d} + \frac1{n^{\omega(1)}}\,,
\end{align*}
 where the last step follows from \prettyref{thm:posterior_mean} under the condition $d \ge (1+\epsilon)n$. 


For $k=3$ and $k=4$, it follows by symmetry and by expansion of $\langle W, W^*\rangle^2$ that
    \begin{align*}
    \mathbb{E}_{G\sim \PRGG}\Big[
    \Big( \mathbb{E}_{X\sim \mu_G} [\mathsf{S}_{3,2}]
    \Big)^2
\Big]
&\lesssim 
\frac{\mathbb{E}_{G\sim \PRGG}\left[\mathbb{E}_{(W,W^*)\sim (\mu_G)^{\otimes 2}}\left[\langle W,W^*\rangle^2\right]\right]}{n^3}
\lesssim \frac{nh(p)^2}{d^2}\,,\\
    \mathbb{E}_{G\sim \PRGG}\Big[
    \Big( \mathbb{E}_{X\sim \mu_G} [\mathsf{S}_{4,2}]
    \Big)^2
\Big]&\lesssim 
\frac{
  \mathbb{E}_{G\sim \PRGG}\left[\mathbb{E}_{(W,W^*)\sim (\mu_G)^{\otimes 2}}\left[\langle W,W^*\rangle^2\right]\right]}{n^4}
    \lesssim \frac{h(p)^2}{d^2}\,,
    \end{align*}
    where the last inequality in each line follows from \prettyref{thm:posterior-overlap-moments} under condition $d\ge (1+\epsilon)n$ and $nh(p)\ge1$.
It follows that
    \begin{align*}
    \binom{n}{4} g(3)_\textup{main}
    &\lesssim \frac{(nh(p))^5}{d^2}
    \log\frac{1}{p}
    =o(1)
    &\textup{provided $p\gtrsim1/n$}\,,\\
    \binom{n}{5} g(4)_\textup{main}
    &\lesssim \frac{(nh(p))^6}{nd^2}
    \ll \frac1n = o(1)\,,
    \end{align*}
again having used the assumption $d \gg (nh(p))^3$.

\paragraph{Case $k=5$.} In this case, on the event $\mathcal{A}_5$ we apply the overall bound on $\eta_5$ given by the final assertion of \prettyref{thm:taylor-eta2345}: this gives
\[
\binom{n}{6} g(5)
\lesssim\frac{n^6p^5}{d^3}
\bigg(\log \frac{1}{p}\bigg)^6
\lesssim\frac{n^{9/2}p^5}{d^{3/2}}
\bigg(\log \frac{1}{p}\bigg)^6
\ll p^{1/2}
\bigg(\log \frac{1}{p}\bigg)^6
= o(1)\,,
\]
provided $d \gtrsim n$ and $d \gg (n h(p))^3$.
\end{proof}





\begin{lemma} \label{lem:partial-sum-ge-6}
Assume that $d\gg (nh(p))^3$, $d \gtrsim n$, and $1/n \lesssim p \ll 1$. Then we have 
$$
\sum_{k = 6}^{n-1} \binom{n}{k+1} g(k) = o(1).
$$
\end{lemma}

\begin{proof}
For $k\ge 6$, the Cauchy--Schwarz inequality gives
\[
g(k)\le \mathbb{E}_{\Pw}[\eta_k(x_1,\dots,x_k)^2]
=\mathbb{E}_{x,y}\Bigg[ 
\mathbb{E}_{x_1,\dots,x_k}\bigg[\prod_{i=1}^k
\kappa(\langle x,x_i\rangle)
\kappa(\langle y,x_i\rangle)
\bigg]\Bigg]
=\mathbb{E}_{x,y}[|\eta_2(x,y)|^k]\,.
\]
As in the proof of \prettyref{lem:partial-sum-2345}, define the event
\[
\mathcal{A}_2 
\equiv \mathcal A_2(x,y)
\equiv \bigg\{|\langle x,y\rangle|+\frac1{d^{1/2}}
\le \frac{c_0}{a^2}\bigg\}\,.
\]
On this event, the expansion of \prettyref{thm:taylor-eta2345} gives
\[
|\eta_2(x,y)|
\lesssim h(p)|\langle x,y\rangle|
+ p \bigg( \log \frac1p \bigg)^2 
\bigg( \langle x,y\rangle^2+\frac{1}{d}\bigg)\,.
\]
The moments of the spherical inner products $\langle x,y\rangle$ can be bounded directly, or using the subgaussian condition (\prettyref{lem:beta.mgf}), yielding
\begin{align*}
\mathbb{E}_{x,y}[|\eta_2(x,y)|^k ; \mathcal{A}_2 ]
&\le (C h(p))^k\mathbb{E}[|W_{12}|^k]
 +(C p(\log(1/p))^2)^k\mathbb{E}[((W_{12})^2+d^{-1})^k] \\
&\le
(C' h(p))^k \frac{k^{k/2}}{d^{k/2}}
+
(C' p(\log(1/p))^2)^k \frac{k^k}{d^k}\,.
\end{align*}
Combining with Stirling's approximation gives
\[
\sum_{k=6}^{n-1}\binom n{k+1}
\mathbb{E}_{x,y}[|\eta_2(x,y)|^k; \mathcal{A}_2]
\lesssim n\cdot \sum_{k=6}^{n-1}
\left[
\left(\frac{C''nh(p)}{\sqrt d}\right)^k
+
\left(\frac{C''np(\log(1/p))^2}{d}\right)^k
\right].
\]
Under $d\gtrsim n$ and $d\gg (nh(p))^3$, the above sum is dominated by the $k=6$ term, which has order
\begin{align*}
n\left(\frac{nh(p)}{\sqrt d}\right)^6
+n\left(\frac{np(\log(1/p))^2}{d}\right)^6
&\lesssim
\left(\frac{(nh(p))^3}{d}\right)^2
+\frac{n^{3/2} p^6(\log(1/p))^{12}}{d^{1/2}} \\
&\ll o(1) + p^{9/2} (\log(1/p))^{21/2}
=o(1) \,,
\end{align*}
where we have used  $d\gtrsim n$ and $d\gg (nh(p))^3$ for the two inequalities respectively.

It remains to control the contribution from the event $(\mathcal A_2)^c$. Here we bound
\[
\sum_{k=6}^{n-1}\binom n{k+1}u^k
\le n(nu)^6\sum_{k=6}^{n-1}\frac{(nu)^{k-6}}{(k+1)!}
\le n(nu)^6 e^{nu}\,.
\]
Since $a^2 \lesssim \log(1/p)$ (\prettyref{lmm:m1m2}), on the event $(\mathcal{A}_2)^c$ we have
$|\langle x,y\rangle|\ge c/\log(1/p)$, so
\begin{align} \nonumber
&\sum_{k=6}^{n-1}\binom n{k+1}
\mathbb{E}_{x,y}[|\eta_2(x,y)|^k; (\mathcal{A}_2)^c] \\
&\qquad\le
Cn\,\mathbb{E}_{x,y}\!\bigg[(n|\eta_2(x,y)|)^6
\exp\{ n|\eta_2(x,y)|\};
|\langle x,y\rangle|\ge \frac{c}{\log(1/p)}
\bigg]\,.\label{eq:kge6bound}
\end{align}
We further decompose \eqref{eq:kge6bound} according to whether $|\langle x,y\rangle|\le 1-\eta$ is a small constant $\eta \in (0,1)$. For
$c/\log(1/p)\le |\langle x,y\rangle|\le 1-\eta$,
\prettyref{lem:eta2-moderate-overlap} gives
$|\eta_2(x,y)|\le C_\eta(\log(1/p))^{\eta/[2(2-\eta)]}p^{\eta/(2-\eta)}$.  
Since $\langle x, y \rangle$ is subaussian with variance proxy $1/d$ (\prettyref{lem:beta.mgf}),  the contribution 
to \eqref{eq:kge6bound} from the range $c/\log(1/p)\le |\langle x,y\rangle|\le 1-\eta$ is
\[
\begin{aligned}
&\le C_\eta n\bigg\{
np^{\eta/(2-\eta)} (\log(1/p))^{\eta/[2(2-\eta)]}\bigg\}^6
\exp\bigg\{
Cn p^{\eta/(2-\eta)} (\log(1/p))^{\eta/[2(2-\eta)]}
-\frac{cd}{\log^2(1/p)}\bigg\} \\
&\le
n^7p^{6\eta/(2-\eta)}\left(\log(1/p)\right)^{3\eta/(2-\eta)}
\exp\bigg\{-\frac{cd}{2 \log^2(1/p)}\bigg\}
=o(1),
\end{aligned}
\]
where the bounds hold because $d\gtrsim n$ and $1/n \lesssim p \ll 1$. 
For the remaining range $1-\eta < |\langle x,y\rangle| \le 1$, we use the Cauchy--Schwarz inequality to obtain the easy bound
\[
|\eta_2(x,y)|
\le
\left(\mathbb E_z[\kappa(\langle z,x\rangle)^2]\right)^{1/2}
\left(\mathbb E_z[\kappa(\langle z,y\rangle)^2]\right)^{1/2}
=1.
\]
Moreover,
$$
\mathbb{P}( \langle x, y \rangle > 1-\eta` )
 = \int_{(1-\eta)\sqrt{d} }^{\sqrt{d}} c_{d,1} \left( 1- \frac{t^2}{d}\right)^{(d-3)/2} \diff t
 \lesssim (2\eta)^{(d-1)/2} /\sqrt{d} .
 $$
Combining the above, we see that the contribution to \eqref{eq:kge6bound} from the range
$1-\eta < |\langle x,y\rangle| \le 1$ is at most
    \[
 C n^7e^n (2\eta)^{(d-1)/2}/\sqrt{d}
    \,.\]
This is $o(1)$ if $d \ge c n$ for any constant $c>0$, by choosing $\eta$ to be a sufficiently small constant depending on $c$. 
\end{proof}

\begin{proof}[\hypertarget{proof.main.thm}{Proof of
\prettyref{thm:main-detection}}]\label{proof:main.thm}
As mentioned above, the main result follows by substituting the bounds from Lemmas~\ref{lem:partial-sum-2345} and \ref{lem:partial-sum-ge-6} into the KL expansion \eqref{eq:KL}.
\end{proof}

\section{Analysis of posterior overlap}
\label{sec:posterior}

\subsection{Bound on posterior mean}
\label{sec:posterior_mean}

In this section we prove 
\prettyref{thm:posterior_mean} on the first moment of the posterior overlap.
Later we extend this result to higher moments using a concentration argument.

We start by recalling some basic results on subgaussianity. A real-valued zero-mean random variable $X$ is called \textbf{subgaussian with variance proxy $\sigma^2$} if
    \[\Expect[e^{tX}] \le \exp\bigg(
        \frac{\sigma^2 t^2}{2}
        \bigg)\]
for all $t\in\mathbb{R}$. An $\mathbb{R}^d$-valued zero-mean random vector $X$ is called \textbf{subgaussian with variance proxy $\sigma^2$} if it holds for every unit vector $v\in\mathbb{S}^{d-1}$ that the scalar random variable $\langle X,v\rangle$ is subgaussian with variance proxy $\sigma^2$~\cite[Section 3.4]{vershynin2018High}.

\begin{lemma}\label{lem:beta.mgf}
Let $x$ is sampled uniformly at random from the unit sphere $\mathbb{S}^{d-1}$. Then $x$ is zero-mean and subgaussian with variance proxy $1/d$.

\begin{proof}
Fix a vector $v\in\mathbb{R}^d$.
Let $t=\|v\|$. Then $Z=\langle v,x\rangle/t$ is a symmetric random variable, and
    \[
    B=Z^2\sim\operatorname{Beta}\bigg(\frac{1}{2},
        \frac{d-1}{2}\bigg)\,.
    \]
It follows from the moment-generating function of the beta distribution that
    \[
    \E_x[\exp(\langle v,x\rangle)]
    = \E [\exp(tZ)]
    = 1 + \sum_{k=1}^\infty \frac{t^{2k}}{(2k)!}
        \prod_{r=0}^{k-1} \frac{1+2r}{d+2r} 
    \le 1 + \sum_{k=1}^\infty \frac{t^{2k}}{k! (2d)^k}
    = \exp\bigg(\frac{\|v\|^2}{2d}\bigg)\,,
    \]
which proves the claim.
\end{proof}
\end{lemma}

\begin{lemma}\label{l:chi.square.lower.tail}
If $X\sim\chi^2(d)$, then it holds for all $x\ge0$ that
    \[
    \mathbb{P}(X \le d-x) \le 
    \exp\bigg(-\frac{x^2}{4d}\bigg)\,.
    \]
\begin{proof}
From the moment-generating function of the chi-square distribution, for any $t<0$ we can bound
    \[
    \mathbb{P}(X \le d-x) 
    \le \frac{\Expect[e^{tX}]}{\exp(t(d-x))}
    = \frac{(1-2t)^{-d/2}}{\exp(t(d-x))}
    \le \exp(tx + dt^2)\,.
    \]
Optimizing over $t<0$ gives the claimed bound.
\end{proof}
\end{lemma}

\begin{lemma}\label{l:subgaus.with.constant}
Suppose $X$ is a real-valued random variable, mean zero, such that
    \[
    \E[e^{tX}] \le c \exp\bigg(\frac{vt^2}{2}\bigg)
    \]
for all $t\in\mathbb{R}$, with $c\ge1$ a constant. Then $X$ is subgaussian with variance proxy $v' \le 2cev$.

\begin{proof}
First note that for large $t$ a subgaussian bound holds, since
    \[
    \E[e^{tX}]
    \le \exp\bigg(\frac{vt^2}{2} + \log c\bigg)
    \le \exp\bigg((1+M)\frac{vt^2}{2}\bigg)
    \]
as long as $|t|\ge h$, where $M>0$ is a parameter to be determined, and 
    \[h \equiv \bigg(\frac{2 \log c}{Mv}\bigg)^{1/2}\,.
    \]
We then consider $|t| \le h$.
For $a,y\in\mathbb{R}$ with $|a|\le1$, we can bound
    \[e^{ay}
    -1-ay
    \le a^2 \sum_{k\ge2}
    \frac{|y|^k}{k!}
    \le a^2 e^{|y|}
    \le a^2 (e^y + e^{-y})\,.
    \]
Since $\E [X]=0$, it follows that, for all $|t| \le h$,
    \[
    \E[e^{tX}]
    \le 1 + \frac{t^2}{h^2}\E[e^{hX} + e^{-hX}]
    \le
    1 + \frac{2}{h^2}
        { c} t^2 \exp\bigg(\frac{vh^2}{2}\bigg)
    = 1 + v t^2 \frac{M  c}{\log c}
    \exp\bigg(\frac{\log c}{M}\bigg)\,.
    \]
Taking $M=\log c$ gives
    \[
    \E[e^{tX}]
    \le 1 + {c} e vt^2 
    \le \exp({c} evt^2)
    \]
for all $|t| \le h$. Combining with the above bound for $|t| \ge h$ shows that $X$ is subgaussian with variance proxy
    \[v' \le 
    \max\{1+\log c, 2ce\} v = 2ce v\,,
    \]
proving the claim.
\end{proof}
\end{lemma}

The proof of \prettyref{thm:posterior_mean} applies an information-theoretic comparison argument. To that end, we recall the 
transportation lemma of Bobkov and G\"otze (see \cite[Lemma 4.18]{BLM13}):
\begin{lemma}
\label{l:transportation}
Under the measure $\mathbb{P}$, suppose $Z$ is a real-valued 
mean-zero
random variable that is subgaussian with variance proxy $\sigma^2$. Then we have
    \[
    |\E_{\mathbb{Q}}[Z]| \le \Big( 2\sigma^2 \operatorname{KL}(
        \mathbb{Q}\|
        \mathbb{P})\Big)^{1/2}
    \]
for any probability measure $\mathbb{Q}$ which is absolutely continuous with respect to $\mathbb{P}$.
\end{lemma}

\begin{corollary}\label{c:transportation}
Suppose $A$ is an $\mathbb{R}^N$-valued subgaussian random vector with mean zero and variance proxy $\sigma^2$. Suppose $B$ is a discrete random variable defined on the same probability space. Then
\begin{equation}
    \Big\| \E[A\,|\,B=b]\Big\|^2
    \le 2\sigma^2
    \KL\Big( \Law(A|B=b) \Big\| \Law(A) \Big)
    \label{eq:condmean-pointwise}
\end{equation}
for any $b$ with $\mathbb{P}(B=b)>0$. Taking expectation over $B$ gives
\begin{equation}
    \E\Big[
    \| \E[A\,|\,B]\|^2
    \Big]
    \le 2\sigma^2
    I(A;B).
    \label{eq:condmean}
\end{equation}
\end{corollary}

\begin{remark}
The inequality \prettyref{eq:condmean} compares the conditional and unconditional means by the KL divergence between the conditional and unconditional distributions (mutual information). This is reminiscent of Tao's inequality \cite[Corollary 7.11]{PW-it}, which states that for any zero-mean $A$ taking values in the interval $(-\sigma,\sigma)$, we have
 $\E[
    (\E[A\,|\,B])^2]
\leq C \sigma^2 I(A;B)$
for some universal constant $C$.
As such, \prettyref{eq:condmean} is a (dimension-free) generalization to subgaussian random vectors.
Results akin to \prettyref{eq:condmean} have previously appeared in the context of generalization bound in machine learning \cite{russo2016controlling,xu2017information}.

Separately, we note  \prettyref{eq:condmean} holds for any random vector $B$,
but it suffices to consider discrete $B$ for the proof. 
Indeed, this follows from discretizing $B$ and applying a continuity argument to the left side and the data processing inequality to the right side.
\end{remark}

\begin{proof}[Proof of Corollary \ref{c:transportation}]
Fix a realization $b$ of $B$ such that $\mathbb{P}(B=b)>0$. For any unit vector $v\in\mathbb{R}^N$, 
applying Lemma~\ref{l:transportation} 
to 
$\Prob = \Law(\Iprod{v}{A})$ 
and $\mathbb{Q} = \Law(\Iprod{v}{A}|B=b)$
yields
\begin{align*}
(\Iprod{v}{\E[A\,|\,B=b]})^2
=
(\E[\Iprod{v}{A}\,|\,B=b])^2
& \leq 2\sigma^2
    \KL(\Law(\Iprod{v}{A}|B=b)\| \Law(\Iprod{v}{A}))\\
& \le 2\sigma^2
    \KL(\Law(A|B=b)\| \Law(A)),    
\end{align*}  
where the last step is by the data processing inequality.
Maximizing the left-hand side over $v$
yields
\[
\Big\| \E[A\,|\,B=b]\Big\|^2
    \le 2\sigma^2
    \KL\Big( \Law(A|B=b) 
    \Big\| \Law(A) \Big)\,,
\]
as claimed. Averaging over $B$ gives the conclusion.
\end{proof}

We aim to apply Corollary \ref{c:transportation} to the inner product matrix $W$. However, $W$ is not subgaussian per se due to its Bernstein-type tail, as given by the Hanson--Wright inequality. Nevertheless, the next lemma shows that $W$ is $O(1/d)$-subgaussian upon truncation on its operator norm.

\begin{lemma}\label{l:W.is.subgaussian}
Let
\begin{equation}
\calH =  \{\|I_n+W\|_{\operatorname{op}}\le 4 \}.
    \label{eq:calH}
\end{equation}
If $d \ge (1+\epsilon) n$ for some arbitrarily small constant $\epsilon>0$, then there exists a constant $c>0$ only depending on $\epsilon$ such that 
$\Prob_0(W \in \calH^c) \leq \exp(-cn)$ 
and,
for any zero-diagonal symmetric matrix $U$,
\[
\mathbb{E}_{X\sim \Pw}[\exp( \langle U,W\mathbf{1}_{\mathcal H}\rangle)]\le \exp\left(\frac{2 \|U\|^2}{d}\right)+1\le 2\exp\left(\frac{2\|U\|^2}{d}\right).
\]

\begin{proof}
Recall that $W=XX^\top -I_n$, so 
$\calH$ is implied by the event
$\{\|X\|_{\rm op} \le 2\}$.
Let $Z$ be an $n\times d$ matrix with i.i.d.\ standard gaussian entries, and let $D$ be the $n\times n$ diagonal matrix with diagonal entries $\|z_1\|,\ldots,\|z_n\|$. Then $X$ is equidistributed as $D^{-1}Z$, so we can bound
    \[\|X\|_{\rm op} \le 
    \frac{\|Z\|_{\rm op}}
        {\min_i \|z_i\|}\,.\]
By the standard concentration inequality for Gaussian matrices~\cite[Cor.~5.35]{vershynin2010introduction}, we have
\[\|Z\|_{\rm op}
\le\sqrt{d} + \sqrt{n}
+t\]
with probability at least $1-2\exp(-t^2/2)$. By the standard concentration bound for chi-squared random variables (Lemma~\ref{l:chi.square.lower.tail}), we have
    \[
    \min\Big\{ \|z_i\|
    : 1\le i\le n\Big\}
    \ge d-(2d)^{1/2} t
    \]
with probability at least $1-n\exp(-t^2/2)$.  Taking $t=cn^{1/2}$, we conclude that with probability at least $1-3n\exp(-c^2 n/2)$, we have
\[\|X\|_{\rm op} \le \frac{1+(1+c)/\sqrt{1+\epsilon}}{[1-\sqrt{2}c/\sqrt{1+\epsilon}]^{1/2}} \le 2\,,\]
where the last inequality holds by choosing $c$ to be a sufficiently small constant  depending on $\epsilon$.

Next, let $W_m$ denote the principal submatrix of $W$ formed by the first $m$ rows and columns. We will truncate on the event
$\mathcal H_m= \{\| I_m+W_m\|_{\operatorname{op}}\le 4 \}$.
Note that $\calH_{m+1} \subset \calH_{m}$ for all $1 \le m \le n-1$ and $\mathcal{H} = \mathcal{H}_n$. Abbreviating $u_n$ for the last column of $U$, we can bound
\begin{align*}
     &\ \mathbb{E}_{X\sim \Pw}\Big[\exp( \langle U,W\mathbf{1}_{\mathcal H}\rangle)\Big]\\
   \le &\ \mathbb{E}_{x_1,\dots,x_n}
   \bigg[\exp\bigg( \sum_{i<j}u_{ij}\langle x_i,x_j\rangle\bigg)
   \mathbf{1}\{W_n \in \mathcal H_n\}\bigg]+1\\
   \le&\ \mathbb{E}_{x_1,\dots,x_{n-1}}\bigg[\exp\bigg(
   \sum_{i<j<n}u_{ij}\langle x_i,x_j\rangle\bigg)\mathbf{1}\{ W_{n-1} \in \mathcal H_{n-1}\} \mathbb{E}_{x_n}\bigg[\exp\bigg(
   \bigg\langle \sum_{i<n} u_{in}x_i,x_n
   \bigg\rangle
   \bigg)\bigg]\bigg]+1\\
    \le&\ \mathbb{E}_{x_1,\dots,x_{n-1}}\bigg[\exp\bigg(\sum_{i<j<n}u_{ij}\langle x_i,x_j\rangle\bigg)\mathbf{1}\{W_{n-1}\in \mathcal H_{n-1}\}\cdot \exp\bigg(\frac{ (u_n)^\top (I_{n-1}+ W_{n-1}) u_n}{2d}\bigg)\bigg]+1\\
    \le&\ \mathbb{E}_{x_1,\dots,x_{n-1}}\bigg[\exp\bigg(\sum_{i<j<n}u_{ij}\langle x_i,x_j\rangle\bigg)\mathbf{1}\{W_{n-1} \in \mathcal H_{n-1}\}\bigg]\cdot \exp\left(\frac{2 \|u_n\|^2
    }{d}\right)+1\,,
\end{align*}
where the second-to-last inequality is by Lemma~\ref{lem:beta.mgf}, and the
 last inequality holds because we have $\|I_{n-1}+W_{n-1}\|_{\operatorname{op} }
\le 4$ on the event $\calH_{n-1}$.
Iterating the above gives the claimed bound.
\end{proof}
\end{lemma}

\begin{proof}[\hypertarget{proof:t.posterior_mean}{Proof of Theorem~\ref{thm:posterior_mean}}]
For $d\ge (1+\epsilon)n$ for some constant $\epsilon>1$, then Lemma \ref{l:W.is.subgaussian} gives $\Pw[W\in \mathcal H]=1-n^{-\omega(1)}$. It follows that 
\begin{align*}
 \mathbb{E}_{G\sim \PRGG}\big[\|\mathbb{E}_{X \sim \mu_G}[W]\|^2]
\le&\ 2\mathbb{E}_{G\sim \PRGG}\big[\|\mathbb{E}_{X \sim \mu_G}[W\mathbf{1}_{\mathcal H}]\|^2\big]+ 2\mathbb{E}_{G\sim \PRGG}\big[\mathbb{E}_{X\sim \mu_G}[\mathbf{1}_{\mathcal H^c}\|W\|^2]\big]\\
\le &\ 2\mathbb{E}_{G\sim \PRGG}\big[\|\mathbb{E}_{X \sim \mu_G}[W\mathbf{1}_{\mathcal H}]\|^2\big]+O(n^{-\omega(1)})\,.
\end{align*}
Lemma~\ref{l:W.is.subgaussian}, combined with Lemma~\ref{l:subgaus.with.constant}, gives that the random vector $W\mathbf{1}_{\mathcal H}$ is subgaussian under the prior $\Pw$ with variance proxy ${16}e/d$.

Furthermore, we claim that $\Expect[W1_\calH]=0$ by symmetry, where we recall the event $\calH$ is defined by  \prettyref{eq:calH}.
Indeed, for any vector of signs $\mathfrak{s}\in \{-1,+1\}^n$ and $D_{\mathfrak{s}} = \diag(\mathfrak{s})$, 
the matrices $W$ and $D_{\mathfrak{s}} W D_{\mathfrak{s}}$ are equidistributed, and also have the same operator norm.
Thus for any pair $i<j$, choosing $\mathfrak{s}_i=-\mathfrak{s}_j$, we have 
\[\Expect[W_{ij} \mathbf{1}_{\{\| I+W\|_{\operatorname{op}} \leq 4\}}] = 
\Expect[(D_{\mathfrak{s}} W D_{\mathfrak{s}})_{ij} \mathbf{1}_{\{\| I+D_{\mathfrak{s}} W D_{\mathfrak{s}}\|_{\operatorname{op}} \leq 4\}}]
= -\Expect[W_{ij} \mathbf{1}_{\{\| I+W\|_{\operatorname{op}} \leq 4\}}]\,,\]
which implies that the above is zero.

We can therefore apply Corollary~\ref{c:transportation} with $A=W\mathbf{1}_{\mathcal H}$ and $B=G$ to obtain
\begin{equation}\label{e:norm.bound.KL.conditional}
    \|\E_{X\sim\mu_G}[W\mathbf{1}_{\mathcal H}]\|^2
    \le \frac{{32}e}{d}
    \KL( \Law(W\mathbf{1}_{\mathcal H}|G) \| \Law(W\mathbf{1}_{\mathcal H}))\,.
    \end{equation}
Taking expectation over $G$ gives 
\[
    \E_{G\sim\PRGG}\Big[
    \|\E_{X\sim\mu_G}[W\mathbf{1}_{\mathcal H}]\|^2
    \Big]
    \lesssim \frac{I(W\mathbf{1}_{\calH};G)}{d}
    \le \frac{H(G)}{d}
    \le \frac{n^2 h(p)}{d}\,,
    \]
    where the last inequality uses \[H(G) \leq \sum_{i < j} H(G_{ij}) = \binom{n}{2}h(p)\,.\]
    This gives the assertion. Finally, we also note that 
    \[
    \KL( \Law(W\mathbf{1}_{\mathcal H}|G) \| \Law(W\mathbf{1}_{\mathcal H}))
    \leq 
    \KL( \Law(W|G) \| \Law(W))  =   \log \frac{1}{\PRGG(G)},
    \]
    using again 
    the data processing inequality and the $W$-measurability of 
    $G$. 
    Thus \eqref{e:norm.bound.KL.conditional} implies    \begin{equation}\label{e:norm.bd.for.later}
    \|\E_{X\sim\mu_G}[W\mathbf{1}_{\mathcal H}]\|^2
    \le \frac{{32}e}{d} \log \frac{1}{\PRGG(G)}
    \le \frac{ {320} e n^2 h(p)}{d}\,,
    \end{equation}
provided $\PRGG(G) \ge \exp(-10 n^2 h(p))$; this will be used below in the proof of Theorem~\ref{thm:posterior-overlap-moments}.
\end{proof}

\subsection{Concentration from log-concavity}
\label{sec:concentration}

We now prove \prettyref{thm:posterior-overlap-moments}, which controls higher moments of the posterior overlap. The main idea is to leverage the strong log-concavity of the posterior distribution of the Gram matrix $W$:

\begin{lemma}\label{l:prior.W.logconcave}
For $d \ge (1+\epsilon)n$, the prior distribution $W\sim\Pw$ conditioned on $\calH$ in \prettyref{eq:calH}
is $(c_0 d)$-strongly log-concave, where $c_0$ is a positive constant depending on $\epsilon$.

\begin{proof}
Under the prior distribution $\Pw$, 
the density of $W$ is given by (see e.g.\ \cite[Sec.~7.6]{anderson2003introduction})
\begin{align*}
\Prob_0(W) 
&\propto e^{-V(W
)} \mathbf{1}\{ I+W \succ 0\}\,, \\
V(W) &= - \frac{d-n-1}{2}\log \det(I+W)\,.
\end{align*}
Abbreviate $\tW=I+W$, and denote by 
$\nabla^2$  the Hessian with respect to the lower triangular elements of $W$. Let $H$ be an $n\times n$ symmetric matrix with zero diagonal. Then
    \begin{align*}
    \log\det(\tW+\epsilon H)
    &=\log\det(\tW)
    +\log\det(I+\epsilon B)\,,\\
    B &= \tW^{-1/2}H\tW^{-1/2}\,.
    \end{align*}
$B$ is a symmetric matrix. Denoting its eigenvalues by $\lambda_1,\ldots,\lambda_n$, we have
    \begin{align*}
    \log\det(\tW+\epsilon H)
    &= \log\det(\tW)
    +\sum_{i=1}^n\log(1+\epsilon\lambda_i) 
    = \log\det(\tW)
    +\epsilon \sum_{i=1}^n \lambda_i
    -\frac{\epsilon^2}{2}
    \sum_{i=1}^n (\lambda_i)^2\\
    &= \log\det(\tW)
    +\epsilon \tr(B)
    -\frac{\epsilon^2}{2}
    \tr(B^2)\,.
    \end{align*}
It follows that
\begin{equation}
-H^\top\nabla^2 \big[\log\det \tilde{W} \big] H =\tr(\tilde{W}^{-1}H\tilde{W}^{-1}H)=\|\tilde{W}^{-1/2}H\tilde{W}^{-1/2}\|_{\operatorname{F}}^2\ge \frac{\|H\|_{\operatorname{F}}^2}{\|\tilde{W}\|_{\operatorname{op}}^2}\,,
\label{eq:hess-prior}
\end{equation}
that is to say, the log-determinant function is strongly log-concave.\footnote{From  \prettyref{eq:hess-prior} we see that 
as long as $d\geq n+1$,
the prior distribution $\Pw$ is always strongly log-concave with constant $\Theta(\frac{d-n-1}{n^2}$); this is tight when all $x_i$'s are aligned. Lemma~\ref{l:prior.W.logconcave} improves the curvature to $\Theta(d)$ by truncating on the typical behavior of $W$.} Next recall from \prettyref{eq:calH} that 
$\calH=\{\|\tilde{W}\|_{\operatorname{op}} \leq 4\}$, which is a convex set in $W$. 
For each $W \in \calH$, the right-hand side of \prettyref{eq:hess-prior} is further lower bounded by $\|H\|^2/16$, and hence
\[
\nabla^2 V(W)\succeq \frac{d-n-1}{16}I_{n(n-1)/2}\succeq c_0 d I_{n(n-1)/2}\,,
\]
for a constant $c_0>0$ depending on $\epsilon$. This proves the claim.
\end{proof}
\end{lemma}

\begin{proof}[\hypertarget{proof:t.posterior-overlap-moments}{Proof of Theorem~\ref{thm:posterior-overlap-moments}}]
Define the event 
\[\mathcal{I} := \bigg\{
\|I_n+W\|_{\operatorname{op}}\le 4
\textup{ and } \|W\|^2\le {2} \frac{n^2}{d}\bigg\}\,.
\]
We also define the $G$-measurable event
\[\mathcal{G}:=
\bigg\{
\PRGG(G)\ge \exp(-10n^2h(p)) \textup{ and }
\mu_G( W \in \mathcal I^c)=n^{-\omega(1)}\bigg\}\,.\]
We will show in Lemma~\ref{eq:rgg-pmf-lower-bound} below that $\PRGG(G\in \mathcal G)=1-n^{-\omega(1)}$; we assume this claim for the remainder of this proof.
For any $G\in \mathcal G$, define $\bar{\mu}_G=\mu_G(\cdot\mid \mathcal I)$. For any fixed $k\ge 1$, we may bound
\begin{align*}
 &\mathbb{E}_{G\sim \PRGG}\left[\mathbb{E}_{(W,W^*)\sim (\mu_G)^{\otimes 2}}\big[|\langle W,W^*\rangle |^{2k}\big]\right] \\
 &\qquad \le\mathbb{E}_{G\sim \PRGG}\left[\mathbf{1}_{\mathcal G}\mathbb{E}_{(W,W^*)\sim \mu_G\otimes \mu_G}\big[|\langle W,W^*\rangle|^{2k}\big]\right]+n^{-\omega(1)}\\
 &\qquad\le 2\mathbb{E}_{G\sim \PRGG}\left[\mathbf{1}_{\mathcal G}\mathbb{E}_{(W,W^*)\sim (\bar{\mu}_G)^{\otimes2}}\big[|\langle W,W^*\rangle|^{2k}\big]\right]+n^{-\omega(1)}\,.
\end{align*}

We next show that for every $G\in \mathcal G$, the inner expectation satisfies the bound
\[
\mathbb{E}_{(W,W^*)\sim (\bar{\mu}_G)^{\otimes2}}\big[\langle W,W^*\rangle^{2k}\big]
\le \bigg(
    O(1) \frac{n^2h(p)}{d}\bigg)^{2k}\,.
\]
To this end, fix a graph $G\in \mathcal G$. We have
\begin{align}\nonumber
&\mathbb{E}_{(W,W^*)\sim (\bar{\mu}_G)^{\otimes2}}\big[\langle W,W^*\rangle^{2k}\big]\\
&\qquad\le 2^{2k} (\mathbb{E}_{(\bar{\mu}_G)^{\otimes2}}[\langle W,W^*\rangle])^{2k}
+2^{2k} \mathbb{E}_{(W,W^*)\sim (\bar{\mu}_G)^{\otimes2}}\big[(\langle W,W^*\rangle-\mathbb{E}_{(\bar{\mu}_G)^{\otimes2}}[\langle W,W^*\rangle])^{2k}\big]\,. \label{e:logconcavity.two.terms}
\end{align}

For the first term on the right-hand side of \eqref{e:logconcavity.two.terms}, since $G\in \mathcal G$, it is easy to see that 
\[
\mathbb{E}_{(W,W^*)\sim (\bar{\mu}_G)^{\otimes2}}[\langle W,W^*\rangle]=\mathbb{E}_{(W,W^*)\sim (\mu_G)^{\otimes2}}[\langle W,W^*\rangle]+n^{-\omega(1)}\,,
\]
and from \eqref{e:norm.bd.for.later} we obtain
\[
\mathbb{E}_{(W,W^*)\sim (\mu_G)^{\otimes2}}[\langle W,W^*\rangle]=\|\mathbb{E}_{W\sim \mu_G}[W]\|^2\le \frac{160 e n^2h(p)}{d}
+ \frac1{n^{\omega(1)}}\,.
\]
(On the right-hand side, the first term is from the contribution of $W\mathbf{1}_\mathcal{H}$, while the second term is from the contribution of $W\mathbf{1}_{\mathcal{H}^c}$, using that $G\in\mathcal{G}$ and $\mathcal{I}\subseteq\mathcal{H}$.)
This implies that the first term on the right-hand side of 
\eqref{e:logconcavity.two.terms} is upper bounded by 
\[\bigg(
O(1) \frac{n^2h(p)}{d} + \frac1{n^{\omega(1)}}\bigg)^{2k}\,.\]

To control the last term in \eqref{e:logconcavity.two.terms}, we use log-concavity. 
The support of $\bar{\mu}_G$ is given by restricting that of $\Prob_0$ to
\[
\left\{\|I_n+W\|_{\operatorname{op}}\le 4\right\}\cap \left\{\|W\|^2\le 2n^2/d\right\}\cap \{W_{ij}\ge \tau,\forall (i,j)\in E(G), W_{ij}< \tau,\forall (i,j)\notin E(G)\}\,,
\]
which is a \textit{convex} set. 
Crucially, a strongly log-concave distribution conditioned on a convex set inherits the 
strong log-concavity with the same constant. Thus, by Lemma~\ref{l:prior.W.logconcave}, the measure $\bar{\mu}_G$ is also $c_0 d$-strongly log-concave and so is $(\bar{\mu}_G)^{\otimes2}$. 
Moreover, the function
\[
(W,W^*)\mapsto \langle W,W^*\rangle
\]
is $\sqrt{2n^2/d}$-Lipschitz continuous in the support of $(\bar{\mu}_G)^{\otimes2}$. The Lipschitz concentration inequality for strongly log-concave measures (see e.g.~\cite[Propns.~5.4.1 and 5.7.1]{bakry2014analysis}) gives that for any $t\ge 0$,
\[
\mathbb{P}_{(W,W^*)\sim (\bar{\mu}_G)^{\otimes2}}\left[ \left|\langle W,W^*\rangle-\mathbb{E}_{\bar{\mu}_G\otimes \bar{\mu}_G}[\langle W,W^*\rangle \right| \geq t\right]\le 2\exp\left( -\frac{c_1 d^2t^2}{n^2}\right)\,,
\]
for some constant $c_1>0,$
meaning that $\langle W,W^*\rangle$ has gaussian-type fluctuations with standard deviation $O(n/d)$. This proves that the last term in 
\eqref{e:logconcavity.two.terms} is upper bounded by
\[ O_k(1) \bigg( \frac{n}{d} \bigg)^{2k}\,,\]
and the claim follows because $n h(p) \ge 1$.
\end{proof}

\begin{lemma} \label{eq:rgg-pmf-lower-bound}
For $d\ge (1+\epsilon)n$, we have $\PRGG(G\in \mathcal G)=1-n^{-\omega(1)}$. 

\begin{proof} Recall that 
    \[\mathcal{I} = 
    \mathcal{H} \cap \bigg\{
        \|W\|^2 \le \frac{2n^2}{d}
        \bigg\}\,,\]
and we already saw in Lemma~\ref{l:W.is.subgaussian} that $\Pw(\mathcal{H}^c) \le n^{-\omega(1)}$. From Lemma~\ref{l:prior.W.logconcave}, the prior distribution $\Pw(\cdot|\calH)$ is $(c_0 d)$-strongly logconcave. The function $W\mapsto\|W\|$ is 1-Lipschitz, and 
\[(\Expect_0[\|W\|])^2 
\le \Expect_0[\|W\|^2] 
\le \frac{n^2}{d}\,,\] so it follows again from the Lipschitz concentration inequality for log-concave measures that 
    \[
    \Pw\bigg(\|W\| \ge \sqrt{\frac{2n^2}{d}}\bigg)
    \leq \Pw\bigg(\|W\| \ge \sqrt{\frac{2n^2}{d}}\bigg|\calH\bigg) + \Pw(\calH^c)
    \le \frac{1}{n^{\omega(1)}}\,.
    \]
This proves $\Pw( W \in \mathcal I^c)=n^{-\omega(1)}$, and it then follows from Markov's inequality that
 \[\PRGG\Big(
 \mu_G( W \in \mathcal I^c)=n^{-\omega(1)}\Big)
 =1-n^{-\omega(1)}\,.\]
It remains to prove that
\[
\PRGG\bigg(\PRGG(G)\ge \exp(-10n^2h(p))\bigg)
\ge 1-n^{-\omega(1)}\,.
\]
For $p=n^{-\Omega(1)}$, this follows from a simple counting argument. First, we have \[\PRGG\bigg(|E(G)|\le \binom{n}{2} \cdot 2p\bigg)=1-n^{-\omega(1)}\,,\] for instance from the coupling between $\PRGG$ and $\PER$ obtained by~\cite[Prop 1.3]{liu2022testing}. The total number of graphs satisfying this edge count constraint is upper-bounded by
    \[
    \exp\bigg\{
    \binom{n}{2} h(2p)
    \bigg\}
    \le \exp\{ 2 n^2 p \log n\}\,.
    \]
As a result, we can bound
    \[
    \PRGG\bigg(\PRGG(G)\le \exp(-10n^2h(p))\bigg)
    \le \frac{\exp\{ 2 n^2 p \log n\}}
        {\exp\{ 10 n^2 p \log n\}}
    \le \frac{1}{n^{\omega(1)}}\,,
    \]
which concludes the proof.
\end{proof}
\end{lemma}

 \section{Concluding remarks}\label{sec:conclusion}

We conclude with a few remarks on the proof techniques and possible extensions.

\paragraph{Comparison of proof techniques.} Following many previous works~\cite{brennan2020phase,liu2022testing,liu2023phase,liu2023probabilistic,mao2026random}, our proof begins with a chain-rule expansion of the KL divergence. Roughly speaking, one reveals the graph sequentially: at step $k$, one exposes the edges between
the new vertex $k$ and the previous vertices $1,\ldots,k-1$. This reduces
the problem to bounding an incremental KL divergence at each step. To obtain a
tight bound, however, one must understand the posterior distribution of these
newly revealed edges given the graph exposed so far, namely the induced graph on
vertices $1,\ldots,k-1$.

Most existing works \cite{liu2022testing,liu2023phase,liu2023probabilistic,bangachev2024detection,bangachev2025random} avoid this posterior analysis by further revealing the latent positions  $x_1, \ldots, x_{k-1}$. While technically convenient, this reveals too much
information and causes the KL bound to be loose. This looseness is precisely why the previous state-of-the-art impossibility result for general $p$ stopped at the condition $d \gg n^3p^2\polylog(n)$~\cite{liu2022testing}. More broadly, this limitation has been
identified in several recent works as a central obstacle to proving tight
detection thresholds for random geometric graph models; see, for example,
\cite[Section 8.1]{bangachev2025random} for a detailed discussion.

Our main departure from this approach is to analyze the posterior distribution
directly. The key type of statement we need is that, when estimating
$W_{12}$, the information contained in the rest of the graph
$A$, beyond the single edge $A_{12}$, is negligible. For soft RGGs, such a result was proved in~\cite{mao2026random} by
approximating the smooth kernel with a polynomial whose coefficients decay rapidly, then
using a change-of-measure argument and controlling the resulting exponential
moments via Hanson--Wright inequalities.

For hard RGGs, this strategy breaks down. First, polynomial
approximations to the indicator kernel do not have sufficiently fast coefficient
decay. Second, and more critically, the change-of-measure argument becomes ineffective, because the posterior measure is uniform over all latent configurations
$X$ that is compatible with the observed graph, rather than a  tilted measure with  weights determined by $W$. The main innovation in our posterior analysis is therefore to
control the posterior overlap by different means: we bound its first moment via
an information-theoretic transportation argument in terms of the entropy of the
random geometric graph, and then control higher moments using the strong
log-concavity of the prior and posterior measures of $W$.

Posterior analysis for hard RGGs has previously been carried out only in the
constant-degree regime $p=\Theta(1/n)$~\cite{liu2022testing}. Although
related in spirit, that analysis has a different nature: it establishes
approximate independence and near-uniformity of $\{x_i:i\in S\}$ under the
posterior for every relatively small subset $S$ of nodes. The proof uses the cavity method and crucially relies
on the locally tree-like structure of neighborhoods in sparse graph, and therefore
appears limited to the constant-degree regime.


\paragraph{Possible extensions.} 
The main ingredients of our proof are quite general. 
The posterior analysis relies primarily on information-theoretic comparison arguments and
strong log-concavity. Therefore, this proof strategy may be
useful for studying many variants of the RGG model such as those considered in~\cite{bangachev2025random,bangachev2024detection,mao2026random}. For instance, we expect the same result to hold  when the latent positions are drawn from an isotropic Gaussian prior. In
that setting, the strong log-concavity (Wishart) continues to hold and the main distribution-specific ingredient that would need to be
rechecked is the expansion of the local interaction function $\eta_k$ in~\prettyref{thm:taylor-eta2345}.

 \paragraph{Challenges when $d \le n$.} Our current proof is restricted to the  $d \ge (1+\epsilon) n$ regime. Establishing the conjectured threshold when $d\le n$ remains an open problem and presents several significant challenges. 
First, although our analysis on the posterior-overlap mean can be extended to
the regime $d=O( n)$, it remains unclear how to prove the sharp bound in \prettyref{thm:posterior_mean} when $d=o(n)$. Second, our current estimate of the posterior-overlap moments in \prettyref{thm:posterior-overlap-moments} crucially relies on the strong log-concavity of the prior and posterior distributions of the Gram matrix $XX^\top$, which holds only for $d>n$. 
Finally, our current analysis controls the contribution of $g(k)$ for large $k$ using a crude estimate that avoids posterior analysis altogether. This approach breaks down for $d=o(n)$ due to rare but significant events that cause the bound to diverge (see Lemma~\ref{lem:partial-sum-ge-6}).

\appendix

\section{Proof of~\prettyref{thm:taylor-eta2345}}
\label{app:etak}

Let $r\ge2$ and $x_1,\ldots,x_r\in \mathbb{S}^{d-1}$ be fixed. Let $W$ be the $r\times r$ zero-diagonal symmetric matrix with  $W_{ij} =\ip{x_i}{x_j}$ for all $1\le i,j\le r$, and let $H \equiv I + W$ be the Gram matrix. Let
    \[\delta=\max_{1 \le i<j \le r}
    \bigg\{
    |W_{ij}|+\frac1{d^{1/2}}
    \bigg\} \,,\]
and assume $\delta \le c_r$ so that $H$ is positive-definite. 

Let $x\sim\Unif(\mathbb{S}^{d-1})$, and consider the random variables
$T_i\equiv \sqrt d\,\ip{x_i}{x}$ for $1\le i\le r$. 
We also let $T$ stand for an independent copy of $T_1$.
Write $\nu(\dd t)=f_d(t)\dd t$ for the law of $T$, whose  density function is given explicitly by
\begin{equation}\label{e:T.density}
f_d(t)
=
c_{d,1}\left(1-\frac{t^2}{d}\right)_+^{(d-3)/2},
\qquad
c_{d,1}
=
\frac{\Gamma(d/2)}
{(\pi d)^{1/2}\Gamma((d-1)/2)}.
\end{equation}
Let $a=\sqrt{d}\tau\ge1$, so that we have  $p=\Pbb(T_i\ge a)$ and $K(t)=\one_{\{t\ge a\}}-p$. Recall that we denote
\[
\kappa(t)=\frac{K(t)}{\sqrt{p(1-p)}}\,.
\]
Let $\nu_H$ denote the joint law of $(T_1,\ldots,T_r)$ as determined by the Gram matrix $H$. Define
\begin{equation}\label{e:def.L.G}
L(t)\equiv \frac{\dd\nu_H}{\dd\nu^{\otimes r}}(t),
\qquad t=(t_1,\ldots,t_r).
\end{equation}
Then we can rewrite the $r$th interaction function as
\begin{equation}
\label{eq:eta-density}
\eta_r(x_1,\ldots,x_r)
=
\E\bigg[\prod_{i=1}^r\kappa(T_i)\bigg]
=
\frac{1}{(p(1-p))^{r/2}}
\int \bigg\{ \prod_{i=1}^r K(t_i) \bigg\}
L(t)\,\nu^{\otimes r}(\dd t).
\end{equation}
Finally put
\begin{align}
m_1=\int tK(t)\,\nu(\dd t)
=p \Expect[T|T\geq a]
= (p(1-p))^{1/2}\zeta\,,
\label{eq:m1}
\end{align}
where we recall that $\zeta$ is defined by \eqref{eq:zeta-def}.
We have 
$m_1\lesssim ap$ and $\zeta\asymp m_1/\sqrt{p} \lesssim a\sqrt{p}$.
(cf.~\prettyref{lmm:m1m2}).

\subsection{An interpolating path}
We define an interpolating model between the correlated projections ($\nu_H$) and independent projections ($\nu^{\otimes r}$). Then we take Taylor approximation of this likelihood ratio and integrate with the kernel to obtain an approximation of $\eta_r$. For $D>r$ and a positive definite $r\times r$ matrix $\Sigma$, define
\[
f_{D,r,\Sigma}(t)
=
\frac{c_{D,r}}{(\det \Sigma)^{1/2}}
\left(1-\frac{t^\top \Sigma^{-1}t}{D}\right)_+^{(D-r-2)/2},
\]
where $c_{D,r}$ is the normalizing constant for the $r$-dimensional
projection of $D^{1/2}x$, with $x$ sampled uniformly at random from $\mathbb{S}^{D-1}$:
\begin{equation}
\label{eq:cDr}
c_{D,r}
= \prod_{j=0}^{r-1} c_{D-j,1}
=
\frac{\Gamma(D/2)}
{(\pi D)^{r/2}\Gamma((D-r)/2)}\,.
\end{equation}
If $x_1,\ldots,x_r$ have Gram matrix $\Sigma$, and $x$ is sampled uniformly at random from $\mathbb{S}^{D-1}$, then $f_{D,r,\Sigma}$ gives the joint density of the random variables
$(D^{1/2}\langle x,x_i\rangle : 1\le i \le r)$.
Let $f_D$ be the corresponding one-dimensional density.  
Now, for $0<\theta\le1$, set
\begin{align*}
    H_\theta
    &=I+\theta(H-I) = I + \theta W\,,\\
D_\theta &= d/\theta^2\,,
\end{align*}
and define the likelihood ratio
\begin{equation}
\label{eq:Ltheta}
L_\theta(t)
=
\frac{f_{D_\theta,r,H_\theta}(t)}
{\prod_{i=1}^r f_{D_\theta}(t_i)}.
\end{equation}
We extend $L_\theta$ continuously to $\theta=0$, where $D_\theta\to\infty$
and $H_\theta\to I$.  In this limit the numerator and denominator both
converge to the corresponding standard Gaussian densities, hence
\[
L_0(t)=1.
\]
At $\theta=1$, $L_1\equiv L$.
For an integer $M\ge0$, define the degree-$M$ Taylor
approximation along this interpolating path:
\begin{equation}
\label{eq:LM-def}
L_M(t)
=
\sum_{j=0}^{M}
\frac{1}{j!}
\left.\frac{\partial^j}{\partial\theta^j}L_\theta(t)\right|_{\theta=0}.
\end{equation}
This gives a decomposition of the interaction function into a Taylor main term and a
Taylor residual: $\eta_r
=
\mathcal M_{r,M}
+
\mathcal R_{r,M}$ where
\begin{align}
\label{eq:main-term-def}
\mathcal M_{r,M}
&:=
\frac{1}{(p(1-p))^{r/2}}
\int \bigg\{ \prod_{i=1}^rK(t_i) \bigg\} L_M(t)\,
\nu^{\otimes r}(\dd t)\,,\\
\label{eq:residual-def}
\mathcal R_{r,M}
&:=
\frac{1}{(p(1-p))^{r/2}}
\int \bigg\{ \prod_{i=1}^rK(t_i)\bigg\}
\Big( L(t)-L_M(t)\Big) \,
\nu^{\otimes r}(\dd t).
\end{align}
This residual term admits the following bound:
\begin{proposition}
\label{prop:generic-residual}
Fix $r\ge2$ and $M\ge1$.  There exist constants $c,C>0$, depending
only on $r$ and $M$, such that if
$1\le a\le \sqrt d/2$ and 
$a^2\delta\le c$, then
\begin{equation}
\label{eq:generic-residual}
|\mathcal R_{r,M}|
\le
C
\left(\frac{p}{1-p}\right)^{r/2}
(a^2\delta)^{M+1}.
\end{equation}
\end{proposition}


Next, we compute the main term
$\mathcal M_{r,M}$ by identifying the degree $M$ of the first nonzero term for each $r$, then apply \prettyref{prop:generic-residual} to prove \prettyref{thm:taylor-eta2345}.

\subsection{Analysis of main term in Taylor expansion}
\label{sec:main-term}

Substituting
\eqref{eq:LM-def} into \eqref{eq:main-term-def}, 
the main term in general is 
    \[\mathcal M_{r,M}
    =
    \sum_{j=0}^{M} \mathcal{M}_r(j)\]
where the $j$-th term of the expansion is 
\begin{equation}
\label{eq:main-term-general}
\mathcal{M}_r(j)
:=
\frac{1}{j!(p(1-p))^{r/2}}
\int
\bigg\{ \prod_{i=1}^r K(t_i)\bigg\} 
\left.
\frac{\partial^j}{\partial\theta^j}L_\theta(t)
\right|_{\theta=0}
\nu^{\otimes r}(\dd t)\,.
\end{equation}
It is straightforward to verify that each Taylor coefficient of $L_\theta(t)$ is a polynomial in $t$. Because $K$ is centered, 
any non-zero contribution to $\mathcal{M}_r(j)$ must come from monomials $t_1^{a_1} \ldots t_r^{a_r}$ with all degrees $a_i\geq 1$. Any such monomial gives a contribution
\[\prod_{i=1}^r \Expect[T_i^{a_i}K(T_i)] = 
\prod_{i=1}^r m_{a_i}\,,\] where 
$m_a \equiv \Expect[T^a K(T)]$,  and $\Expect$ denotes expectation over $T\sim\nu$.  As we will see below, the dominant contributions come from individual degrees $a_i\in\{1,2\}$. We note the following lemma, which gives a simple relation between $m_1$ and $m_2$:

\begin{lemma}
\label{lmm:m1m2}
With the above notations, we have
    \[
    m_2 = \frac{d-1}{d}am_1
    \]
and $1 \le a \le ( 2 \log (1/p))^{1/2}$.
Furthermore, assuming $a\le \sqrt d/2$, we have
\[ \zeta^2 \asymp pa^{2} \lesssim h(p)\,,
\]
where we recall that $\zeta$ is defined by \eqref{eq:zeta-def}.

\begin{proof}
In this proof, we fix $d$ and suppress it from the notation.  We abbreviate $\mathbb{P}=\nu$, and write $\Expect$ for expectation with respect to $\nu$.  We also abbreviate $c\equiv c_{d,1}$ and $f(t)\equiv f_d(t)$. Define
    \[ g(t)
    = c \bigg(1-\frac{t^2}{d}\bigg)^{(d-1)/2}\,,
    \]
and note that
    \begin{align*}
    \frac{d}{dt} g(t) &= -\frac{d-1}{d}  t f(t)\,,\\
    \frac{d}{dt} [tg(t)]
        &= -(t^2-1) f(t)\,.
    \end{align*}
It follows that
    \begin{align*}
    &m_1
    = \Expect[T;T\ge a]
    = \int_a^{d^{1/2} } t f(t)\,dt
    = \frac{d}{d-1} g(a)
    = \frac{cd}{d-1} \bigg(1-\frac{a^2}{d}
        \bigg)^{(d-1)/2}\,,\\
    &\E[T^2; T \ge a]-p
    =\int_a^{d^{1/2}} (t^2-1) f(t)\,dt
    =ag(a)
    = \frac{d-1}{d}am_1\,.
    \end{align*}
Rearranging the last expression gives
    \[
    m_2
    = \E[T^2 K(T)]
    = \E[T^2;T\ge a]-p
    = \frac{d-1}{d}am_1\,,
    \]
which proves the first claim. Next, recalling that $p=\mathbb{P}(T\ge a)$, to bound $a\ge1$ it suffices to prove that $\mathbb{P}(T\ge 1)> p$. In fact it is easy to argue that $\mathbb{P}(T\ge 1)$ is lower bounded by an absolute constant: $T$ is equidistributed as
    \[
    \frac{d^{1/2} Z}{(Z^2 + S)^{1/2}}
    \]
where $Z$ is a standard gaussian random variable, and $S$ is a chi-square random variable with $d-1$ degrees of freedom, independent of $Z$. If $Z \ge 2^{1/2}$ and $S \le 2(d-1)$ then the above ratio is $\ge1$, so we conclude
    \[
    \mathbb{P}(T\ge1)
    \ge \mathbb{P}(Z \ge 2^{1/2}) \mathbb{P}(S \le 2(d-1))
    \ge \frac{\mathbb{P}(Z \ge 2^{1/2})}{2}
    \ge 0.03\,.
    \]
This proves $a\ge1$, as claimed. On the other hand the upper bound on $a$ is equivalent to the bound 
\[p = \mathbb{P}(T\ge a) \le \exp\bigg(-\frac{a^2}{2}\bigg)\,,\] which follows from the fact that $T$ is $1$-subgaussian (Lemma~\ref{lem:beta.mgf}). Next, recalling the definition \eqref{eq:zeta-def} of $\zeta$, it is clear that $\zeta\gtrsim p^{1/2}a$. For the upper bound, we recall that the integration by parts above gives
    \[m_1 = \frac{d}{d-1} f_d(a) 
    \bigg(1-\frac{a^2}{d}\bigg)
    \leq 2 f_d(a)\,.
    \]
It is shown later in 
Lemma~\ref{l:cap-tail-moment}
that under the assumption $1\le a\le \sqrt d/2$ we have
$p\geq c f_d(a)/a$ for some numerical constant $c$ (see display \eqref{eq:p-lb}). It  follows that
\[\zeta^2 = \frac{m_1^2}{p(1-p)}
\lesssim \frac{(ap)^2}{p} = pa^2
\lesssim h(p)\,,
\] as claimed.
\end{proof}
\end{lemma}

The following lemma provides a uniform bound on $\eta_2(x_1,x_2)$ whenever $\langle x_1,x_2\rangle\le 1-\epsilon$ for $\epsilon \in (0,1)$, and is used in the proof of \prettyref{lem:partial-sum-ge-6}.

\begin{lemma}\label{lem:eta2-moderate-overlap}
Assume that $d\gg (nh(p))^3$ and $1/n \lesssim p\ll 1$.  If fixed vectors 
$x_1,x_2\in\mathbb S^{d-1}$ satisfy $\langle x_1,x_2\rangle\le 1-\epsilon $ for any $\epsilon\in (0,1)$, then there is a constant $C>0$ such that 
\[
|\eta_2(x_1,x_2)|\le C
p^{\epsilon/(2-\epsilon)}\bigg(\log\frac1p\bigg)^{\epsilon/[2(2-\epsilon)]}
\,.
\]
\end{lemma}

\begin{proof}
Take $a=d^{1/2}\tau$ and $T$ as before, so $p=\mathbb{P}(T\ge a)$. 
Recall the probability density function given in~\eqref{e:T.density}.
We use similar calculations as in the proof of \prettyref{lmm:m1m2}. It is straightforward to verify that $c_{d,1}\asymp1$.
It follows from the assumptions that $1 \le a \ll d^{1/4}$. For $t\in[a,a+1/a]$, we have
    \[
    f_d(t) \ge c_{d,1} \bigg(
    1-\frac{a^2+2+1/a^2}{d}\bigg)^{(d-1)/3}
    \ge c' f_d(a)\,,
    \]
where $c'$ is an absolute constant. It follows that
    \[
    p=\mathbb{P}(T\ge a)
    \ge \int_{a}^{a+1/a} f_d(t)\,dt
    \ge \frac{c' f_d(a)}{a}\,.
    \]
On the other hand, with $g$ is as in the proof of \prettyref{lmm:m1m2}, we have
    \begin{align*}
    \mathbb{P}\bigg(T\ge \sqrt{\frac{2}{2-\epsilon}} a\bigg)
    &= \int_{(2/(2-\epsilon))^{1/2}a}^{d^{1/2}} f_d(t)\,dt
    \le \frac{((2-\epsilon)/2)^{1/2}}{a} \int_{(2/(2-\epsilon))^{1/2}a}^{d^{1/2}} t f_d(t)\,dt \\
    &=\frac{d}{d-1}
    \frac{((2-\epsilon)/2)^{1/2}}{a} g\bigg( \sqrt{\frac{2}{2-\epsilon}} a\bigg)
    \le \frac{C_\epsilon }{a} \bigg(1-\frac{2a^2}{(2-\epsilon)d}\bigg)^{(d-1)/2}\,.
    \end{align*}
Applying the bound $1-\alpha x \le (1-x)^{\alpha}$ for $\alpha>1$ gives
    \[
    \mathbb{P}\bigg(T\ge \sqrt{\frac{2}{2-\epsilon}} a\bigg)
    \le  \frac{C_\epsilon }{a}  \bigg(1-\frac{a^2}{d}\bigg)^{(d-3)/(2-\epsilon)}
    \le C_\epsilon a^{\epsilon/(2-\epsilon)} p^{2/(2-\epsilon)}\,,
    \]
where the last inequality follows from the previous lower bound on $p$.


Now suppose $x_1,x_2\in\mathbb{S}^{d-1}$ with $\langle x_1,x_2\rangle\le 1-\epsilon$. If $\langle x,x_i\rangle\ge\tau$ for both $i=1,2$, then
    \[\bigg\langle x, \frac{x_1+x_2}{\|x_1+x_2\|}\bigg\rangle
    \ge \frac{2\tau}{\|x_1+x_2\|}
    = \frac{2\tau}{[2(1+\langle x_1,x_2\rangle)]^{1/2}}
    \ge \sqrt{\frac{2}{2-\epsilon}}\tau\,.
    \]
It follows from the preceding bounds that
    \[
    \mathbb{P}_x(\langle x,x_1\rangle\ge\tau,\langle x,x_2\rangle\ge\tau)
    \le \mathbb{P}_x\bigg(
        \bigg\langle x, \frac{x_1+x_2}{\|x_1+x_2\|}\bigg\rangle
        \ge \sqrt{\frac{2}{2-\epsilon}}\tau
        \bigg) \le C_\epsilon a^{\epsilon/(2-\epsilon)} p^{2/(2-\epsilon)}\,.
    \]
Since $a^2 \lesssim \log(1/p)$ by \prettyref{lmm:m1m2}, we conclude
    \[
|\eta_2(x_1,x_2)|
\le
\frac{
\mathbb P_x\{\langle x,x_1\rangle\ge\tau, \langle x,x_2\rangle\ge\tau\}
+p^2}{p(1-p)}
\lesssim p^{\epsilon/(2-\epsilon)}\bigg(\log\frac1p\bigg)^{\epsilon/[2(2-\epsilon)]}\,,
\] as claimed.
\end{proof}


We refer to  \eqref{eq:logL-explicit} for the log-likelihood ratio $\Lambda_\theta(t)=\log L_\theta(t)$. 
We now formally expand this expression in $\theta$. For the normalizing constants, we use the asymptotic expansion of the log-gamma function to obtain
    \[\log \frac{c_{D_\theta,r}}{(c_{D_\theta,1})^2}
    =\log \bigg\{ 
    \frac{\Gamma(D_\theta/2) \Gamma((D_\theta-1)/2)^{r-1}}{\Gamma(D_\theta/2)^r} \bigg\} 
    = -\frac{r(r-1)}{4d}\theta^2 + O(\theta^4)\,.
    \]
Since $H_\theta=I+\theta W$,
we can expand
    \begin{align*}
    \frac12 \log \det H_\theta
    &= \frac12 \log \bigg\{
    1 - \theta^2 \sum_{i<j} W_{ij} W_{ji}
        + 2 \theta^3  \sum_{i<j<k} W_{ij} W_{jk} W_{ki}  + O(\theta^4)
    \bigg\} \\
    &= -\theta^2 \frac{\tr(W^2)}{4}
    + \theta^3 \frac{\tr(W^3)}{6}
    + O(\theta^4)\,.
    \end{align*}
For the remaining terms of $\log L_\theta(t)$, direct calculation gives
    \begin{align*}
    \frac{(D_\theta-3)}{2} \log \bigg(
        1-\frac{(t_i)^2}{D_\theta}
        \bigg)
    &= -\frac{(t_i)^2}{2}
    - \theta^2\bigg\{
    \frac{(t_i)^4}{4d}
    -\frac{3(t_i)^2}{2d}
    \bigg\} + O(\theta^4)\,,
    \\
    \frac{D_\theta-r-2}{2}
    \log \bigg(1 - 
    \frac{\langle t,(H_\theta)^{-1} t\rangle}{D_\theta}
    \bigg)
    &= -\frac{\|t\|^2}{2}
    + \theta \frac{\langle t, Wt\rangle}{2}
    + \theta^2 \bigg\{
    - \frac{\langle t, W^2 t\rangle}{2}
    + \frac{(r+2)\|t\|^2}{2d} 
    - \frac{\|t\|^4}{4d}
    \bigg\}\\
    &\qquad + \theta^3 \bigg\{
    -\frac{(r+2) \langle t,W t\rangle}{2d}
    + \frac{\langle t,W^3 t\rangle }{2}
    + \frac{\|t\|^2 \langle t,W t\rangle }{2d}
    \bigg\}
    + O(\theta^4)\,.
    \end{align*}
Combining the above calculations gives
\[
\Lambda_\theta(t) \triangleq \log L_\theta(t)
=
\theta B(t)
+\theta^2 B_2(t)+\theta^3 B_3(t)+O(\theta^4),
\]
where $B(t) \equiv \langle t,Wt\rangle/2$, 
    \begin{align*}
    B_2(t)
    &=
    \frac{\tr(W^2)}4
    -\frac{\langle t, W^2t\rangle } 2
    -\frac{1}{4d}\bigg\{\|t\|^4-\sum_{i=1}^r (t_i)^4\bigg\}
    +\frac{r-1}{2d}\|t\|^2
    -\frac{r(r-1)}{4d}\,, \\
    B_3(t)
    &= - \frac{\tr(W^3)}{6}
    -\frac{(r+2) \langle t,W t\rangle}{2d}
    + \frac{\langle t,W^3 t\rangle }{2}
    + \frac{\|t\|^2 \langle t,W t\rangle }{2d}\,.
    \end{align*}
Note that every monomial in $B_2$ involves at most two coordinates $i\in[r]$, while every monomial in $B_3$ involves at most three coordinates.
It follows that
\begin{equation}\label{e:formal.expansion.L.theta}
L_\theta(t)
=
1+\theta B(t)
+\theta^2\left\{\frac12 B(t)^2+B_2(t)\right\}
+\theta^3\left\{\frac16 B(t)^3+B(t)B_2(t)+B_3(t)\right\}
+O(\theta^4)\,. 
\end{equation}
Thus, for $r=3$ and $r=4$, the only second-order terms that can survive
centering come from $B(t)^2/2$.  For $r=5$, the only third-order terms that
can depend on all five coordinates come from $B(t)^3/6$, since
$B(t) B_2(t)$ involves at most four coordinates and $B_3(t)$ involves at most
three.


\begin{proof}[\hypertarget{proof:t.taylor-eta2345}{Proof of \prettyref{thm:taylor-eta2345}}]
Recall \eqref{eq:main-term-def},
\eqref{eq:residual-def}, and \eqref{eq:main-term-general}.
Since $K(T_i)$ is centered, we have $\mathcal{M}_r(0)=0$ for all $r\ge1$. For $r=2$, we take $M=1$.  The residual estimate \prettyref{prop:generic-residual} gives
\[
|\mathcal R_{2,1}|
\le
C\frac{p}{1-p}a^4\delta^2\,.
\]
For the main term, recalling \eqref{e:formal.expansion.L.theta},
the only contribution to $\mathcal{M}_{2,1}$ comes from $j=1$:
    \[
    \mathcal{M}_{2,1}
    =\mathcal{M}_2(1)
    =
    \int
    [W_{12} t_1 t_2]
    \prod_{i=1}^2 \kappa(t_i) f_d(t_i) \,dt_i
    = \frac{W_{12} m_1^2}{p(1-p)}
    =\zeta^2 W_{12}\,,
    \]
where we recall that $\zeta$ is defined by \eqref{eq:zeta-def}. This proves the assertion for $r=2$.

For $r=3$, we take $M=2$. Then \prettyref{prop:generic-residual} gives
\[
|\mathcal R_{3,2}|
\le
C\left(\frac{p}{1-p}\right)^{3/2}a^6\delta^3.
\]
Recalling \eqref{e:formal.expansion.L.theta},
each monomial in the $j=1$ coefficient involves at most two coordinates, so we have $\mathcal{M}_3(1)=0$. In the $j=2$ term, the only monomials involving all three coordinates come from $B(t)^2/2$: thus
    \begin{align*} 
    \mathcal{M}_{3,2}
    & =\mathcal{M}_3(2)
    = 
    \int 
    \Big\{
    W_{12}W_{13}t_1^2t_2t_3
+
W_{12}W_{23}t_1t_2^2t_3
+
W_{13}W_{23}t_1t_2t_3^2
\Big\} \prod_{i=1}^3 \kappa (t_i) f_d(t_i)\,dt_i
    \\
    &= \frac{(m_1)^2 m_2}{(p(1-p))^{3/2}}
    \Big\{ W_{12}W_{13}
+
W_{12}W_{23}
+
W_{13}W_{23}\Big\} \\
    &= \zeta^2 
    \frac{d-1}{d} a \zeta 
        \Big\{ W_{12}W_{13}
+
W_{12}W_{23}
+
W_{13}W_{23}\Big\}
    \end{align*}
This proves the assertion for $r=3$.

For $r=4$, we again take $M=2$. Then \prettyref{prop:generic-residual} gives
\[
|\mathcal R_{4,2}|
\le
C\left(\frac{p}{1-p}\right)^2a^6\delta^3\,.
\]
By similar reasoning as above, we have $\mathcal{M}_4(1)=0$. In the $j=2$ term, the only monomials involving all four coordinates come from $B(t)^2/2$, and involve pairs of disjoint edges:
    \begin{align*}
    \mathcal{M}_{4,2}
    & =\mathcal{M}_4(2)
    =\int \Big\{
    W_{12}W_{34}t_1t_2t_3t_4
+
W_{13}W_{24}t_1t_2t_3t_4
+
W_{14}W_{23}t_1t_2t_3t_4
    \Big\}
    \prod_{i=1}^4 \kappa(t_i) f_d(t_i)\,dt_i \\
    &= \frac{(m_1)^4}{(p(1-p))^2}
    \Big\{ W_{12}W_{34}
+
W_{13}W_{24}
+
W_{14}W_{23}\Big\}
    = \zeta^4 \Big\{ W_{12}W_{34}
+
W_{13}W_{24}
+
W_{14}W_{23}\Big\}\,.
    \end{align*}
This proves the assertion for $r=4$.

For $r=5$, we take $M=3$. Then \prettyref{prop:generic-residual} gives
\[
|\mathcal R_{5,3}|
\le
C\left(\frac{p}{1-p}\right)^{5/2}a^8\delta^4\,.
\]
By similar reasoning as above, we have $\mathcal{M}_5(1)=\mathcal{M}_5(2)=0$.
For the $j=3$ term, the only monomials involving all five coordinates come from $B(t)^3/6$, and involve choosing three distinct edges covering all the vertices $\{1,\ldots,5\}$.
Such a graph has one vertex of degree two and one disjoint edge on the two
remaining vertices, and the sum of the corresponding edge products is exactly
$\mathsf{S}_{5,3}$.  The associated monomial has one coordinate squared and the
other four coordinates to the first power, so we conclude
\[
\mathcal M_{5,3}
=\mathcal M_5(3)
= \frac{m_2(m_1)^4}{(p(1-p))^{5/2}} \mathsf{S}_{5,3}
=\frac{d-1}{d} a \zeta^5 \mathsf{S}_{5,3}\,.
\]
This proves the assertion for $r=5$.
Moreover, for $r=5$, if we choose $M=2$, then the main term in the Taylor expansion vanishes, i.e., $\mathcal M_{5,2}=0$.
Indeed, the $j=0,1,2$ terms vanish after centering, because no monomial there can
depend on all five coordinates.
This, together with Proposition~\ref{prop:generic-residual}, then implies
\[
|\eta_5(x_1,\ldots,x_5)|=|\mathcal R_{5,2}|\le
C\left(\frac{p}{1-p}\right)^{5/2}(a^2\delta)^3\,,
\]
which concludes the proof.
\end{proof}

\subsection{Bound on residual term in Taylor expansion}
\label{sec:residual-bound}

In this section we prove \prettyref{prop:generic-residual}.
Throughout this section write
\[
U(t)=1+\sum_{i=1}^r t_i^2 .
\] For what follows, we recall the product rule
    \begin{equation}\label{e:deriv.product.rule}
    \bigg(\frac{d}{dx}\bigg)^m
    [f(x) g(x)]
    = \sum_{\ell=0}^m
    \binom{m}{\ell} f^{(\ell)}(x)
        g^{(m-\ell)}(x)\,.
    \end{equation}
We also recall Fa\`a di Bruno's chain rule:
    \begin{equation}\label{e:faa.di.bruno}
    \bigg(\frac{d}{dx}\bigg)^m
    f(g(x))
    = \sum_{\underline{k}}
    \frac{m!}{(k_1)! \cdots (k_m)!}
    f^{(k)}(g(x))
    \prod_{j=1}^m
    \bigg(\frac{g^{(j)}(x)}{j!}\bigg)^{k_j}\,,
    \end{equation}
where the sum is over all nonnegative tuples $\underline{k}\equiv (k_1,\ldots,k_m)$ satisfying
    \[
    \sum_{j=1}^m j k_j=m\,,
    \]
and we abbreviate $k\equiv k_1 + \ldots + k_m$. Now, recalling  \eqref{eq:cDr}, we denote
\begin{equation}\label{e:rewrite.gamma.factors}
\Gamma_{r,d}(\theta)
:=\log \frac{c_{D_\theta,r}}
    {(c_{D_\theta,1})^r}
=\log \bigg\{
\frac{\Gamma((D_\theta-1)/2)^r }
    {\Gamma((D_\theta-r)/2) 
    \Gamma(D_\theta/2)^{r-1}}
\bigg\}
\,.
\end{equation}
The logarithm of \eqref{eq:Ltheta} can be written as
\begin{equation}
\label{eq:logL-explicit}
\Lambda_\theta(t):=\log L_\theta(t)
\equiv \textup{I}_\theta(t)+\textup{II}_\theta(t)
+\textup{III}_\theta(t)\,,
\end{equation}
where $\Lambda_0\equiv 0$ by continuity, and we decompose
\begin{align} \nonumber 
\textup{I}_\theta(t)
&\equiv \Gamma_{r,d}(\theta)
-\frac12\log\det H_\theta\,,\\ \nonumber 
\textup{II}_\theta(t)
&\equiv \frac{d}{2\theta^2}
\bigg\{ 
\log\bigg(1-\frac{\theta^2}{d}
\langle t, (H_\theta)^{-1}t \rangle \bigg)
-\sum_{i=1}^r\log\bigg(1-\frac{\theta^2}{d}t_i^2\bigg)
\bigg\}\,,\\
\textup{III}_\theta(t)
&\equiv -\frac{r+2}{2}
\log\bigg(1-\frac{\theta^2}{d}
\langle t, (H_\theta)^{-1}t \rangle
\bigg)
+\frac32\sum_{i=1}^r\log\bigg(1-\frac{\theta^2}{d}t_i^2\bigg)\,.
\label{eq:logL-rewritten}
\end{align}
We will refer to this decomposition in some of the analysis that follows.

\begin{lemma}\label{l:log.likelihood.local.region.bounds}
Fix $r\ge2$ and $N\ge1$. For $\Lambda_\theta=\log L_\theta$ as defined by \eqref{eq:logL-explicit}, we have the bounds
\begin{align}
\label{eq:log-size-bound}
\sup\bigg\{|\Lambda_\theta(t)|
:0\le\theta\le1
\bigg\}
&\le
C_{r}\delta U(t)\,,\\
\label{eq:log-derivative-bound}
\sup\bigg\{ \bigg|
\frac{\partial^j}{\partial\theta^j}\Lambda_\theta(t)
\bigg| : 0 \le \theta \le 1\bigg\}
&\le
C_{r,j}[\delta U(t)]^j \,,
\end{align}
for all $U(t)\le c_0/\delta$, where the second line holds for all $1\le j\le N$.
\end{lemma}

\begin{proof}
On the region $U(t)\le c_0/\delta$, all logarithmic arguments in
\eqref{eq:logL-explicit} stay bounded away from zero. We now proceed to analyze separately each of the terms in the decomposition~\eqref{eq:logL-rewritten}:



\paragraph{Term I.}
Recall that $H_\theta=I+\theta W$. Since
$\norm{W}_{\operatorname{op}}\le C_r\delta$, $(H_\theta)^{-1}$ has operator
norm bounded by $C_r$. Let $\lambda_i$ ($1\le i\le r$) be the eigenvalues of $W$. Then, for $j\ge1$, we have
    \begin{align*}
    \frac{\partial^j}{\partial \theta^j}
    \log \det H_\theta
    &= 
    \frac{\partial^j}{\partial \theta^j}
    \sum_{i=1}^r \log (1+\theta\lambda_i)\\
    &= (-1)^{j-1}(j-1)!
    \sum_{i=1}^r
    \bigg(\frac{\lambda_i}{1+\theta\lambda_i}
    \bigg)^j 
    = (-1)^{j-1}(j-1)!
    \tr[((H_\theta)^{-1} W)^j]\,,
    \end{align*}
using that $H_\theta$ and $W$ commute.
Therefore the determinant term satisfies
\begin{align*}
& |\log\det H_\theta|\le C_r\delta\,, \\
&\max_{1\le j\le N} \bigg|
\frac{\partial^j}{\partial\theta^j}
\log\det H_\theta
\bigg|
\le C_{r,j}\delta^j \le C_{r,j} [\delta U(t)]^j\,.
\end{align*}
We now turn to the normalization term $\Gamma_{r,d}(\theta)$. The shifted log-gamma Stirling expansion
(see 
\cite[\href{https://dlmf.nist.gov/5.11.E8}{(5.11.8)}]{NIST:DLMF} and 
\cite[Eq.~(2) and Theorem~3]{Nemes2013HermiteGamma})
gives
    \begin{equation}\label{e:shifted.log.gamma.expansion}
    H_u(z) 
    \equiv \log\Gamma(z-u)-\log\Gamma(z)+u\log z
    =H_{u,\le N}(z) + H_{u,>N}(z)\,,
    \end{equation}
where $H_{u,\le N}(z)$ gives the expansion up to order $N$, and $H_{u,>N}(z)$ is the remainder term:
    \begin{align*}
    H_{u,\le N}(z)
    &\equiv \sum_{m=1}^N \frac{a_m(u)}{z^m}\,,\\
    H_{u,>N}(z)
    &= O_{r,N}\bigg(\frac1{z^{N+1}}\bigg)\,,
    \end{align*}
where the last estimate holds uniformly over $0\le u\le r/2$. Moreover, the same estimate holds after differentiating up to order $N$, that is to say,
    \[
    \frac{d^\ell}{dz^\ell} H_{u,>N}(z)
    = O_{r,N}\bigg(\frac{1}{z^{N+1+\ell}}\bigg)\,,
    \]
again uniformly over $0\le u\le r/2$. Then, by the chain rule \eqref{e:faa.di.bruno},
    \[
    \bigg|\frac{d^m}{ds^m}
    H_{u,>N}\bigg(\frac{1}{2s}\bigg)\bigg|
    \le \sum_{\underline{k}}
    \frac{m! O_{r,N}(s^{N+1+m})}{(k_1)! \cdots (k_m)!}
    \prod_{j=1}^m \bigg( \frac{1}{s^{j+1}}\bigg)^{k_j}
    \le C_{r,N}
    \sum_{k=1}^m \frac{s^{N+1+m}}{s^{m + k}}
    \le C_{r,N} s^{N+1-m}\,,
    \]
uniformly over $0\le s\le 1/d$ and $1 \le m \le N$. 
Meanwhile, $H_{u,\le N}(1/(2s))$ is a polynomial in $s$, and is easy to differentiate directly with respect to $s$. It follows that
    \begin{equation}\label{e:H.u.s.derivative.bound}
    \bigg|
    \frac{d^m}{ds^m}H_u\bigg(\frac{1}{2s}\bigg)\bigg|
    \le 
    \bigg|
    \frac{d^m}{ds^m}H_{u,\le N}
        \bigg(\frac{1}{2s}\bigg)\bigg|
    +    \bigg|
    \frac{d^m}{ds^m}H_{u,>N}\bigg(\frac{1}{2s}\bigg)\bigg|
    \le C_{r,N}\,,
    \end{equation}
uniformly over $0\le s\le 1/d$ and $1 \le m \le N$.
Now, recalling \eqref{e:rewrite.gamma.factors}, we set $z\equiv D_\theta/2 \equiv 1/(2s)$, and rewrite
\begin{align} \nonumber
\Gamma_{r,d}(\theta)
&=
r\bigg\{\log\Gamma(z-1/2)-\log\Gamma(z)+\frac12\log z\bigg\}
-
\bigg\{\log\Gamma(z-r/2)-\log\Gamma(z)+\frac r2\log z\bigg\} \\
&= rH_{1/2}(z)-H_{r/2}(z) \equiv F_r(s)\,.
\label{eq:Gamma}
\end{align}
From the expansion \eqref{e:shifted.log.gamma.expansion}, we have
    \[
|\Gamma_{r,d}(\theta)|
= |F_r(s)|
\le C_r s \le \frac{C_r}{d} \le C_r\delta U(t)\,,
\]
uniformly over $0 \le \theta\le 1$, using that $\delta \ge d^{-1/2}$. From the estimate \eqref{e:H.u.s.derivative.bound}, we have
\[
\bigg|
\frac{d^m}{ds^m} F_r(s) \bigg|\le C_{r,N}\,,
\]
uniformly over $0\le s\le 1/d$ and $1 \le m \le N$. Another application of the chain rule \eqref{e:faa.di.bruno} gives
    \[
    \bigg|\frac{\partial^j}{\partial \theta^j}\Gamma_{r,d}(\theta)\bigg|
    =\bigg|\frac{\partial^j}{\partial \theta^j} F_r\bigg(\frac{\theta^2}{d}\bigg)
    \bigg|
    \le \sum_{k_1,k_2}
    \frac{j!}{(k_1)!(k_2)!}
    \bigg|(F_r)^{(k)}\bigg(\frac{\theta^2}{d}\bigg)\bigg|
    \bigg(\frac{2\theta}{d}\bigg)^{k_1}
    \bigg(\frac{1}{d}\bigg)^{k_2}\,,
    \]
where the sum goes over $k_1,k_2\ge0$ with $k_1+2k_2=j$ and $k_1+k_2=k$. It follows that $\lceil j/2 \rceil \le k \le j$, and therefore
    \[
    \bigg|\frac{\partial^j}{\partial \theta^j}\Gamma_{r,d}(\theta)\bigg|
    \le \frac{C_{r,j}}{d^{\lceil j/2 \rceil}}
    \le C_{r,j} \delta^j
    \le C_{r,j} [\delta U(t)]^j\,,
    \]
for all $1\le j\le N$. Altogether, this verifies that the term $\textup{I}_\theta$ satisfies the bounds
\eqref{eq:log-size-bound} and \eqref{eq:log-derivative-bound}.




\paragraph{Term III.} 
Let $A_\theta(t)\equiv \langle t,(H_\theta)^{-1}t\rangle$.
Recall that $\|W\|_\textup{op} \le C_r\delta$ and $\|(H_\theta)^{-1}\|_\textup{op} \le C_r$. As before, let $\lambda_i$ ($1\le i\le r$) denote the eigenvalues of matrix $W$, and now let $(u_1,\ldots,u_r)$ be the corresponding orthonormal eigenbasis. Since the matrices  $(H_\theta)^{-1}$ and $W$
commute, we can express
    \begin{align*}
    \frac{\partial^j}{\partial\theta^j}(H_\theta)^{-1}
    &=\frac{\partial^j}{\partial\theta^j}
    \sum_{i=1}^r \frac{1}{1+\theta\lambda_i}
    (u_i)(u_i)^\top\\
    &=\sum_{i=1}^r (-1)^j j!
    \frac{1}{1+\theta\lambda_i}
    \bigg(\frac{\lambda_i}{1+\theta\lambda_i}
    \bigg)^j (u_i)(u_i)^\top
    = (-1)^jj!
    (H_\theta)^{-1}[W(H_\theta)^{-1}]^j\,.
    \end{align*}
It follows that, for $0\le j\le N$, we have
\[
    \bigg|
    \frac{\partial^j}{\partial \theta^j} A_\theta(t)
    \bigg|
 \le
\bigg\|\frac{\partial^j}{\partial\theta^j}(H_\theta)^{-1}
\bigg\|_{\operatorname{op}}
\|t\|^2
\le C_{r,j}\delta^jU(t)\,. \]
Recall that $\delta \ge d^{-1/2}$. If we define $Z_\theta(t) \equiv \theta^2 A_\theta(t)/d$, then for $U(t) \le c_0/\delta$ we have
    \[|Z_\theta(t)|
    \le \frac{|A_\theta(t)|}{d}
    \le C_r \delta^2 U(t)
    \le C_r c_0\delta \le \frac12\,.\]
Applying the product rule \eqref{e:deriv.product.rule} gives 
    \[
    \bigg|\frac{\partial^j}{\partial\theta^j}
    Z_\theta(t)\bigg|
    = \frac{1}{d} \bigg| \frac{\partial^j}{\partial\theta^j}
    [ \theta^2 A_\theta(t) ] \bigg|
    =\frac1d \bigg| \sum_{\ell=0}^j
    \binom{j}{\ell}
    \frac{\partial^\ell}{\partial\theta^\ell}
    (\theta^2)
    \frac{\partial^{j-\ell}}{\partial\theta^{j-\ell}}
    A_\theta(t)\bigg|
    \le \frac{1}{d} C_{r,j} \delta^{j-2} U(t)
    \le C_{r,j} \delta^j U(t)\,,
    \]
since the only non-zero contributions to the above sum come from $0\le \ell\le 2$. It is straightforward to verify that the estimates of the last two displays are also valid if we replace $Z_\theta(t)$ with $\theta^2 (t_i)^2/d$ for any $1\le i\le r$. From this we deduce
    \[
    |\textup{III}_\theta(t)|\le C_r\delta^2U(t)\le C_r\delta U(t)\,.\]
By the chain rule \eqref{e:faa.di.bruno},
for $1\le j\le N$ we obtain
    \[\bigg|
    \frac{\partial^j}{\partial\theta^j}
    \log (1-Z_\theta(t))\bigg|
    \le C_{r,j} \sum_{\underline{k}}
    \prod_{\ell=1}^j
    (\delta^\ell U(t))^{k_\ell}
    \le C_{r,j} \delta^j
    \sum_{k=1}^j U(t)^k
    \le C_{r,j} [\delta U(t)]^j\,.
    \]
Altogether, this verifies that the term $\textup{III}_\theta$ satisfies the bounds
\eqref{eq:log-size-bound} and \eqref{eq:log-derivative-bound}.

\paragraph{Term II.} Recall that we continue to restrict ourselves to the region $U(t) \le c_0/\delta$. In this regime, the expansion
    \begin{equation}
\textup{II}_\theta(t)
=
-\frac12\left(A_\theta(t)-\|t\|^2\right)
-\sum_{k=2}^\infty 
\frac{1}{2k}
\underbrace{\bigg(\frac{\theta^2}{d}\bigg)^{k-1}
\bigg\{A_\theta(t)^k
- \sum_{i=1}^r (t_i)^{2k}
\bigg\}}_{\equiv T_{k,\theta}(t)}.
\label{eq:II}
\end{equation}
is uniformly convergent. For the first term, we note that
    \[
\bigg|A_\theta(t)-\sum_{i=1}^r (t_i)^2\bigg|
= \Big| \langle t,
[(H_\theta)^{-1}-I]t \rangle \Big|
\le C_r \delta U(t)\,.
\]
Moreover, we recall from the preceding calculations (for the $\textup{III}_\theta$ term) that
    \[\bigg|
    \frac{\partial^j}{\partial\theta^j}
    \bigg(
    A_\theta(t)-\sum_{i=1}^r (t_i)^2
    \bigg)\bigg|
    \le C_{r,j} \delta^j U(t) 
    \le C_{r,j} [\delta U(t)]^j\,.\]
Again recalling that $U(t) \le c_0/\delta$, we can bound
    \[
    \sum_{k=2}^\infty |T_{k,\theta}(t)|
    \le\sum_{k=2}^\infty C_r U(t)
    \bigg( \frac{C_r U(t)}{d}\bigg)^{k-1}
    \le \sum_{k=2}^\infty C_r U(t)
    (C_r\delta)^{k-1}
    \le C_r \delta U(t)\,.
    \]
We next proceed to bound the derivatives of $T_{k,\theta}(t)$ (with respect to $\theta$). To this end, we first apply the chain rule \eqref{e:faa.di.bruno} to bound 
    \[
    \bigg|\frac{\partial^\ell}{\partial\theta^\ell}
    [A_\theta(t)^k]\bigg|
    \le \sum_{\underline{q}}\bigg| 
    \frac{\ell!}{q_1! \cdots q_\ell!}
    (k)_q A_\theta(t)^{k-q}
    \prod_{j=1}^\ell
    \bigg(
    \frac{(\partial_\theta)^j A_\theta(t)}{j!}
    \bigg)^{q_j}\bigg|\,,
    \]
where the sum goes over nonnegative tuples $q_1,\ldots,q_\ell$ with $q_1 + \ldots + q_\ell=q$ and $q_1 \cdot 1 + \ldots + q_\ell \cdot \ell = \ell$. Simplifying gives 
    \[\bigg|\frac{\partial^\ell}{\partial\theta^\ell}
    [A_\theta(t)^k]\bigg|
    \le C_{r,\ell} k^\ell 
    \sum_{\underline{q}}
    [C_r U(t)]^{k-q}
    \prod_{j=1}^\ell 
    (C_{r,j} \delta^j U(t))^{q_j}
    \le C_{r,\ell} ({C_r} U(t))^k
    (k \delta)^\ell 
    \]
We conclude that for all $\ell\ge0$,
    \[
    \bigg|\frac{\partial^\ell}{\partial\theta^\ell}
    \bigg( A_\theta(t)^k
    - \sum_{i=1}^r (t_i)^{2k}\bigg)\bigg|
    \le C_{r,\ell} ({C_r} U(t))^k
    (k \delta)^\ell \,.
    \]
Then, applying the product rule \eqref{e:deriv.product.rule} gives 
    \begin{align}\nonumber
    \bigg|
    \frac{\partial^j}{\partial\theta^j}
    T_{k,\theta}(t)\bigg|
    &\le \frac{1}{d^{k-1}}\sum_{\ell=0}^j
    \binom{j}{\ell} \bigg| 
    \frac{\partial^\ell}{\partial\theta^\ell}
    (\theta^{2k-2})
    \cdot
    \frac{\partial^{j-\ell}}
    {\partial\theta^{j-\ell}}
    \bigg(A_\theta(t)^k-\sum_{i=1}^r (t_i)^{2k}
    \bigg) \bigg|  \\
    &\le \frac{C_{r,j}}{d^{k-1}}
    \sum_{\ell=0}^{\min\{j, 2k-2\}}
    k^{\ell} \theta^{2k-2-\ell}
    [C_r U(t)]^k (k\delta)^{j-\ell} \nonumber \\
    &\le \frac{C_{r,j} ({C_r} U(t))^k k^{j}
        \delta^{\max\{0,j-2k+2\}} }{d^{k-1}}
    \le C_{r,j} [C_r U(t)]^k k^j
         \delta^{\max\{2(k-1),j\}}\,.
    \label{eq:Tk-derivative-term}
    \end{align}
We now sum \eqref{eq:Tk-derivative-term} over $k\ge2$: for $k\le j$, we can simply bound
    \[\bigg|
    \frac{\partial^j}{\partial\theta^j}
    T_{k,\theta}(t)\bigg|
    \le C_{r,j} [C_r U(t)]^k \delta^j
    \le C_{r,j} [\delta U(t)]^j\,.
    \]
We bound the sum over $k >j$ by
    \[\sum_{k=j+1}^\infty
    \bigg|\frac{\partial^j}{\partial\theta^j}
    T_{k,\theta}(t)\bigg|
    \le C_{r,j}
    \sum_{k=j+1}^\infty
    [C_r U(t)]^k k^j \delta^{2(k-1)}
    \le C_{r,j} [C_r U(t)] [C_r U(t)\delta^2]^j
    \le C_{r,j} [\delta U(t)]^j\,.
    \]
Altogether we conclude
\[
\sum_{k\ge2}\bigg|
\frac{\partial^j}{\partial \theta^j} T_{k,\theta}(t)\bigg|
\le
C_{r,j}[\delta U(t)]^j\,.
\]
Combining with the calculations for the first term in \eqref{eq:II} gives $|\textup{II}_\theta(t)|\le C_r\delta U(t)$, and
\[
\bigg|\frac{\partial^j}{\partial\theta^j} \textup{II}_\theta(t)
\bigg|
\le C_{r,j}[\delta U(t)]^j
\]
for all $1\le j\le N$. Altogether, we have now verified that the terms 
$\textup{I}_\theta$, $\textup{II}_\theta$, and $\textup{III}_\theta$
from the decomposition~\prettyref{eq:logL-rewritten} all satisfy the bounds
\eqref{eq:log-size-bound} and \eqref{eq:log-derivative-bound}. This concludes the proof.
\end{proof}

\begin{lemma}\label{l:cap-tail-moment}
Recall that $T$ is equidistributed as $d^{1/2} x_1$, where $x$ is sampled uniformly at random from $\mathbb{S}^{d-1}$.
If $1 \le a^2 \le d/4$, 
then we have
\begin{equation}
\label{eq:cap-tail-moment}
\Expect\Big[
|K(T)|(1+T^2)^m e^{\lambda T^2}\Big]
\le
C_m p a^{2m}
\end{equation}
for all $\lambda\le c/a^2$, and for each fixed
integer $m\ge0$.
\end{lemma}
\begin{proof}
We have $|K(t)|\le \one\{t\ge a\}+p$, so
    \begin{equation}
    \label{e:K.T.power.tail.bound}
    \Expect
    \Big[
    |K(T)|(1+T^2)^m e^{\lambda T^2}\Big]
    \le
    p \bigg(\Expect
    \Big[(1+T^2)^m e^{\lambda T^2}\Big]
    + 
    \Expect
    \Big[(1+T^2)^m e^{\lambda T^2}
    \,\Big|\, T\ge a\Big]\bigg)\,.
    \end{equation}
For the first term on the above right-hand side, we recall from Lemma~\ref{lem:beta.mgf} that $T$ is $1$-subgaussian, which implies
    \[\Expect
    \Big[(1+T^2)^m e^{\lambda T^2}\Big]
    \le C_m\,.\]
For the second term on the right-hand side of 
\eqref{e:K.T.power.tail.bound}, we will argue that
\begin{equation}
 \Expect\Big[
 (1+T^2)^m e^{\lambda T^2}\,\Big|\,T\ge a\Big] 
 \leq C_m a^{2m}\,,
    \label{eq:cap-tail-moment2}
\end{equation}
which implies the desired bound \prettyref{eq:cap-tail-moment}. Recall that $T$ has density $f_d(t)$ given by \eqref{e:T.density}.
Indeed, the
assumptions $a^2\delta\le c$ and $\delta\ge d^{-1/2}$ imply (loosely) that
$a^2\le d/4$.
Fix $c_1>0$ a small constant.
On the interval
$[a,a+c_1/a]$, we have
$t^2-a^2\le 2c_1+c_1^2$, and
\[
\frac{f_d(t)}{f_d(a)}
=
\left(1-\frac{t^2-a^2}{d-a^2}\right)^{(d-3)/2}
\ge
\exp\left(
-\frac{d(t^2-a^2)}{d-a^2}
\right)
\ge c_0\,,
\] 
from which it follows that
\begin{equation}
p=\int_a^{d^{1/2}}f_d(t)\,\dd t
\ge
\int_a^{a+c_1/a}f_d(t)\,\dd t
\ge c_0 c_1 \frac{f_d(a)}{a}
\equiv c_2 \frac{f_d(a)}{a}\,.
    \label{eq:p-lb}
\end{equation}
On the other hand, we have $(d-3)/(2(d-a^2))\ge 1/8$, so for all $t \ge a$ we can bound
\[
\frac{f_d(t)}{f_d(a)}
=
\bigg(1-\frac{t^2-a^2}{d-a^2}\bigg)^{(d-3)/2}
\le
\exp\bigg(-\frac{(t^2-a^2)}{8}\bigg)\,.
\]
This allows us to bound
\begin{align*}
\Expect\Big[(1+T^2)^m e^{\lambda T^2}
\,\Big|\, T\ge a\Big]
&=\frac{1}{p}\int_a^{\sqrt d}(1+t^2)^m e^{\lambda t^2}f_d(t)\,\dd t \\
&\le
\frac{f_d(a)}{p} e^{\lambda a^2}
\int_a^\infty (1+t^2)^m e^{-(1/8-\lambda)(t^2-a^2)}\,\dd t.
\end{align*}
Let $y=t^2-a^2$.  Since
$1/8-\lambda\ge 1/8 - c/a^2 \ge 1/16$ and
$\dd t=\dd y/(2\sqrt{a^2+y})\le \dd y/(2a)$, the above can be upper bounded by
\begin{align*}
\frac{f_d(a)}{p} e^{\lambda a^2}
\frac1{2a}
\int_0^\infty (1+a^2+y)^m e^{-y/16}\,\dd y
&\le\frac{f_d(a)}{p} e^{\lambda a^2}
 (2a)^{2m-1}
\int_0^\infty (1+y)^m e^{-y/16}\,\dd y \\
&\le
\frac{f_d(a)}{p} e^{\lambda a^2}
C_m a^{2m-1}
\le C_m a^{2m}\,,
\end{align*}
where the last step uses the lower bound 
\prettyref{eq:p-lb} together with the assumption $\lambda \le c/a^2$.
This proves the claim \prettyref{eq:cap-tail-moment2}, and as we noted above the desired bound \prettyref{eq:cap-tail-moment} follows. 
\end{proof}

\begin{lemma}[local Taylor derivative bound]
\label{lmm:taylor-derivative}
Fix $r\ge2$ and $N\ge1$.  There exist constants $c_0,c,C>0$, depending
only on $r$ and $N$, such that the following holds.  If $1 \leq a^2 \leq d/4$,
then
\begin{equation}
\label{eq:derivative-integrated}
\int_{U(t)\le c_0/\delta}
\bigg\{ \prod_{i=1}^r |K(t_i)|\bigg\}
\sup_{0\le\theta\le1}
\bigg|
\frac{\partial^N}{\partial\theta^N}L_\theta(t)
\bigg|
\nu^{\otimes r}(\dd t)
\le
C p^r (a^2\delta)^N \,.
\end{equation}
\end{lemma}

\begin{proof}
Recall the bounds for $\Lambda_\theta$ from Lemma~\ref{l:log.likelihood.local.region.bounds}. We now convert these to bounds for $L_\theta=\exp(\Lambda_\theta)$: by the chain rule \eqref{e:faa.di.bruno}, for $1\le m\le N$,
\[\bigg|
\frac{\partial^m}{\partial\theta^m}
L_\theta(t)\bigg|
=\bigg|
L_\theta(t)
\sum_{\underline{\alpha}}
\frac{m!}{\alpha_1! \cdots \alpha_m!}
\prod_{q=1}^m
\bigg(
\frac{(\partial_\theta)^q\Lambda_\theta(t)}{q!}
\bigg)^{\alpha_q}\bigg|\,,
\]
where the sum goes over non-negative tuples $\underline{\alpha}=(\alpha_1,\ldots,\alpha_m)$
with $\alpha_1 + 2 \alpha_2 + \ldots + m \alpha_m=m$. Applying Lemma~\ref{l:log.likelihood.local.region.bounds} gives
\begin{equation}\bigg|
\frac{\partial^m}{\partial\theta^m}
L_\theta(t)\bigg|
\le
C_m L_\theta(t)
\sum_{\underline{\alpha}} 
\prod_{q=1}^m \Big( C_{r,q}[\delta U(t)]^q\Big)^{\alpha_q}
\le C_{r,m}
\exp\{ C_r \delta U(t)\}
[\delta U(t)]^m
\le C_{r,m} [\delta U(t)]^m\,.
\label{eq:L-derivative-pointwise}
\end{equation}
for $1\le m \le N$ and $U(t) \le c_0/\delta$.


It remains to integrate \eqref{eq:L-derivative-pointwise}. To this end, note that
\[
U(t)^N
\le
C_{r,N}\sum_{\underline{m}}
\prod_{i=1}^r(1+t_i^2)^{m_i}] \, ,
\] where the sum goes over nonnegative tuples $(m_1,\ldots,m_r)$ with $m_1 + \ldots + m_r = N$. Then, combining with \eqref{eq:L-derivative-pointwise}, the left-hand side of \eqref{eq:derivative-integrated} is
    \begin{align*}
    &\le C_{r,N} \delta^N 
    \sum_{\underline{m}}
    \int_{U(t) \le c_0/\delta}
    \bigg\{\prod_{i=1}^r |K(t_i)|
    (1+(t_i)^2)^{m_i}
    \bigg\}  \nu^{\otimes r}(\dd t)\\
    &\le C_{r,N} \delta^N
    \sum_{\underline{m}}
    \prod_{i=1}^r \E[|K(T)|(1+T^2)^{m_i}]
    \le C_{r,N} p^r (\delta a^2)^N\,,
    \end{align*}
where the last bound is by Lemma~\ref{l:cap-tail-moment}. This proves the claimed bound 
\eqref{eq:derivative-integrated}.
\end{proof}

To curb the tail contribution, we also need a global control on the Taylor polynomial without requiring $U(t) \leq c_0/\delta$.
\begin{lemma}
\label{lmm:taylor-polynomial-growth}
Fix $r\ge2$ and $M\ge1$.  There are constants $c,C_{r,M}>0$ such that,
if $\delta\le c$, then
\[
|L_M(t)|\le C_{r,M}U(t)^M
\]
for all $t$, where we recall that $L_M$ is defined by
\eqref{eq:LM-def}.
\end{lemma}

\begin{proof}
Recall from \eqref{eq:LM-def} that
    \[L_M(t)
    = \sum_{j=0}^M
    \frac1{j!}
\left.
\frac{\partial^j}{\partial\theta^j}L_\theta(t)
\right|_{\theta=0}
    \equiv 
    \sum_{j=0}^M
    Q_j(t)\,.\]
We claim that for $0\le j\le M$, the $\theta^j$-coefficient $Q_j(t)$
is a polynomial in $t$ of degree at most $2j$, whose coefficients are polynomials 
in $W$ and $d^{-1}$, with
numerical coefficients depending only on $r$ and $j$.  
This implies
$|Q_j(t)|\le C_{r,j} U(t)^j$, and the assertion follows.

It remains to justify the polynomial-coefficient claim. We address separately each of the terms in the  decomposition \eqref{eq:logL-rewritten}:
\begin{itemize}
\item First consider $\textup{I}_\theta$, which does not depend on $t$.  Recall that $H_\theta=I+\theta W$.  The determinant part satisfies
\[
\log\det H_\theta
=
\tr\log(I+\theta W)
=
\sum_{n=1}^{M}(-1)^{n+1}\frac{\tr(W^n)}{n}
\theta^n
+O_{r,M}(\theta^{M+1}),
\]
with Taylor coefficients being polynomial in $W$.
For the normalizing factor, it follows from  \prettyref{eq:Gamma}, along with the expansion of $H_u(z)$ given in Lemma~\ref{l:log.likelihood.local.region.bounds}, that
\[
\Gamma_{r,d}(\theta)
=
\sum_{m=1}^{M}
\bigg(\frac{2\theta^2}{d}\bigg)^m
\bigg\{
ra_m(1/2)-a_m(r/2)\bigg\}
+O_{r,M}\left(\frac{\theta^{2M+2}}{d^{M+1}}\right)\,.
\]
Thus the $\theta^j$-coefficient of $\textup{I}_\theta$ is a polynomial in $W$ and $d^{-1}$.

\item 
Next consider $\textup{II}_\theta$. Recall \prettyref{eq:II}, namely,
\[
\textup{II}_\theta(t)
=
-\frac12\bigg(A_\theta(t)-\sum_i t_i^2\bigg)
-\sum_{k\ge2}
\frac{1}{2k}
\bigg(\frac{\theta^2}{d}\bigg)^{k-1}
\bigg\{A_\theta(t)^k-\sum_i t_i^{2k}\bigg\}\,,\]
where we recall that
    \[A_\theta(t)=\langle t,(H_\theta)^{-1}t\rangle
    = \sum_{n\ge0} (-\theta)^n
    \langle t,W^n t \rangle
    \,.\]
In particular, the $\theta^j$-coefficient of $A_\theta(t)$ is a quadratic polynomial in $t$, whose coefficients with respect to $t$ are polynomial in $W$.
Consequently, the first term in $\textup{II}_\theta$ has $\theta^j$-coefficients
that are quadratic in $t$, hence of degree at most $2j$ for $j\ge1$. For the term in 
$\textup{II}_\theta$ involving the sum over $k\ge2$, note that a contribution to the
$\theta^j$-coefficient requires $2k-2\le j$.  Then necessarily $j\ge2$, and
$2k\le j+2\le2j$.  Since the factor in braces has degree at most $2k$ in
$t$, hence at most $2j$, and its coefficients with respect to $t$
are polynomials in $W$ and $d^{-1}$.

\item 
Finally consider $\textup{III}_\theta$.  Unlike $\textup{II}_\theta$, there is no prefactor
$d/\theta^2$.  Expanding the logarithms directly,
\begin{align*}
\textup{III}_\theta(t)
&=
\frac{r+2}{2}\sum_{k\ge1}\frac1k
\left(\frac{\theta^2}{d}\right)^k A_\theta(t)^k
-\frac32\sum_{i=1}^r\sum_{k\ge1}\frac1k
\left(\frac{\theta^2}{d}\right)^k t_i^{2k}.
\end{align*}
A contribution to the
$\theta^j$-coefficient therefore requires $2k\le j$.  Its degree in $t$ is at
most $2k$, hence at most $j$, and its coefficients are polynomials in $W$ and $d^{-1}$.

\end{itemize}
Combining the three parts, we conclude that the $\theta^j$ coefficient of $\Lambda_\theta(t)$ is 
\[
    \lambda_j(t)
    := \frac{1}{j!} 
    \frac{\partial^j}{\partial\theta^j}
    \Lambda_\theta(t)\bigg|_{\theta=0}\,,
    \]
where $\lambda_j(t)$ is a polynomial in $t$ of degree at most $2j$, and its
coefficients with respect to $t$ are polynomials in $W$ and $d^{-1}$. Applying the chain rule \eqref{e:faa.di.bruno} gives
    \[
    Q_j(t)
    = \sum_{\underline{k}}
    \frac{j!}{k_1!\cdots k_j!}
    \prod_{s=1}^j \lambda_s(t)^{k_s}\,,
    \]
where the sum goes over nonnegative tuples $(k_1,\ldots,k_j)$ with $k_1 + 2 k_2 + \ldots + jk_j=j$. This implies the polynomial coefficient claim for $Q_j(t)$. \end{proof}


\begin{lemma}[tail contribution]
\label{lmm:tail-contribution}
Fix $r\ge2$ and $M\ge1$.  With $c_0$ as in
\prettyref{lmm:taylor-derivative}, there exist constants $c,C>0$,
depending only on $r$ and $M$, such that if
$a\geq 1$ and $a^2\delta\le c$,
then
\begin{equation}
\label{eq:tail-contribution-exponential}
\int_{\{U(t)>c_0/\delta\}}
\bigg\{\prod_{i=1}^r |K(t_i)|\bigg\}
\Big\{
L(t)+|L_M(t)|
\Big\}
\nu^{\otimes r}(\dd t)
\le
C\exp\{-c/\delta\}.
\end{equation}
\end{lemma}

\begin{proof}

Let $U_*=c_0/\delta$. Applying the union bound and the marginal subgaussian tail,
\[
\nu^{\otimes r}(t : U(t) \ge U_*)
\le
\sum_{i=1}^r
\nu\bigg(
t_i : |t_i| \ge \bigg( \frac{U_*-1}{r}\bigg)^{1/2}
\bigg)
\le
C_r e^{-c_rU_*}\,.
\]
A similar bound holds if we include a polynomial factor $U(t)^M$:
\begin{equation}
\label{eq:product-tail-large-U}
\int_{\{U>U_*\}}U(t)^M\,\nu^{\otimes r}(\dd t)
\le
C_{r,M}e^{-c_rU_*}\,.
\end{equation}
Since $|K|\le1$, combining with \prettyref{lmm:taylor-polynomial-growth} gives
\[
\int_{\{U>U_*\}}
\bigg\{ \prod_{i=1}^r
|K(t_i)|\bigg\}
|L_M(t)|\,\nu^{\otimes r}(\dd t)
\le
C_{r,M}e^{-c_rU_*}\,.
\]
Next, recall from \eqref{e:def.L.G} that $L$ is defined as the likelihood ratio between $\nu_H$ and $\nu^{\otimes r}$, where $\nu_H$ is the joint law of $(T_1,\ldots,T_r)$. Under $\nu_H$, each $T_i$ has the same marginal law $\nu$ as under $\nu^{\otimes r}$. A similar union bound thus gives
\begin{align*}
&\int_{\{U>U_*\}}
\bigg\{\prod_{i=1}^r |K(t_i)|
\bigg\}L(t)\,\nu^{\otimes r}(\dd t)
\leq 
\int_{\{U>U_*\}}
 L(t)\,\nu^{\otimes r}(\dd t)
=
\nu_H\{U>U_*\} \\
&\qquad\le \sum_{i=1}^r
\nu\{|T_i|>\sqrt{(U_*-1)/r}\}
\le
C_re^{-c_rU_*}\,,
\end{align*}
thereby proving the claim.
%
\end{proof}

\begin{proof}[Proof of \prettyref{prop:generic-residual}]
Divide $\mathbb{R}^r$ into the regions
\begin{align*}
\Omega_0 &:=\{t:U(t)\le c_0/\delta\}\,,\\
\Omega_1 &:= \{t:U(t)>c_0/\delta\}\,.
\end{align*}
Recalling
\eqref{eq:residual-def}, this gives a corresponding decomposition 
$\mathcal{R}_{r,M}=\mathcal{R}_{r,M}(\Omega_0)
+\mathcal{R}_{r,M}(\Omega_1)$.
On $\Omega_0$, the whole path $\theta\mapsto L_\theta(t)$ is smooth, and
Taylor's theorem gives
\[
L(t)-L_M(t)
=
\frac{1}{M!}
\int_0^1
(1-\theta)^M
\frac{\partial^{M+1}}{\partial\theta^{M+1}}L_\theta(t)\,\dd\theta .
\]
Therefore, the $\Omega_0$-contribution to $\mathcal R_{r,M}$ satisfies
\[
\begin{split}
|\mathcal R_{r,M}(\Omega_0)|
&\le
\frac{1}{[p(1-p)]^{r/2}}
\int_{\Omega_0}
\bigg\{\prod_{i=1}^r |K(t_i)|\bigg\}
|L(t)-L_M(t)|
\nu^{\otimes r}(\dd t)
\\
&\le
\frac{1}{[p(1-p)]^{r/2}}
\int_{\Omega_0}
\bigg\{\prod_{i=1}^r |K(t_i)|\bigg\}
\frac{1}{M!}
\sup_{0\le\theta\le1}
\bigg|
\frac{\partial^{M+1}}{\partial\theta^{M+1}}L_\theta(t)
\bigg|
\nu^{\otimes r}(\dd t).
\end{split}
\]
Applying \prettyref{lmm:taylor-derivative} yields
\[
|\mathcal R_{r,M}(\Omega_0)|
\le
C_{r,M}\bigg(\frac{p}{1-p}\bigg)^{r/2}(a^2\delta)^{M+1}.
\]
For the $\Omega_1$-contribution, 
\prettyref{lmm:tail-contribution} gives
\begin{align*}
|\mathcal R_{r,M}(\Omega_1)|
&\le
\frac{1}{[p(1-p)]^{r/2}}
\int_{\Omega_1}
\bigg\{\prod_{i=1}^r |K(t_i)|\bigg\}
\{L(t)+|L_M(t)|\}
\nu^{\otimes r}(\dd t)\\
& \leq 
\frac{C_{r,M}\exp\{-c/\delta\}}{[p(1-p)]^{r/2}}
\le
C_{r,M}'\left(\frac{p}{1-p}\right)^{r/2}(a^2\delta)^{M+1}.
\end{align*}
where the last step follows because $\frac{1}{\delta} \gtrsim a^2 \gtrsim \log \frac{1}{p}$ (up to constants depending on $r,M$), thanks to the assumption $a^2\delta \leq c$ and \prettyref{eq:p-lb}.
Combining the two pieces completes the proof.
\end{proof}


\section*{Declaration of AI use}
AI assistance was used in developing and drafting the proof of Theorem~2, as
well as in editing the exposition of the paper. All mathematical arguments,
statements, and final text were reviewed, verified, and approved by the authors.

\section*{Acknowledgment}
H.~Du and N.~Sun are supported in part by NSF-Simons collaboration grant DMS-2031883 and NSF grant DMS-2347177.
C.~Mao is supported in part by NSF CAREER Award 2338062.  J.~Xu is supported in part by NSF CAREER Award CCF-2144593. Part of this research was conducted during J.~Xu's visit to the MIT Institute for Data, Systems, and Society in Spring 2026, hosted by Victor Chernozhukov.

\bibliographystyle{alpha}
\bibliography{ref,etak_new_refs}

\end{document}